\documentclass[12pt,a4paper]{article}
\usepackage{etex}
\usepackage[utf8]{inputenc}
\usepackage[TS1,T1]{fontenc}
\usepackage{xspace}
\usepackage[a4paper,tmargin=3truecm,bmargin=3truecm,rmargin=2.2truecm,lmargin=2.2truecm]{geometry}

\usepackage[all]{xy}

\usepackage[english]{babel}
\usepackage{lmodern}

\usepackage{enumitem}
\newlist{enum-hypothesis}{enumerate}{1}
\setlist[enum-hypothesis]{label=(\arabic*),itemsep=0pt, parsep=0pt}

\usepackage{amsfonts,amssymb,amsmath,mathtools,slashed}
\usepackage[thmmarks,amsmath]{ntheorem}

\usepackage{hyperref}

\newtheorem{theorem}{Theorem}[section]
\newtheorem{proposition}[theorem]{Proposition}
\newtheorem{lemma}[theorem]{Lemma}

\theoremstyle{remark}
\theoremsymbol{$\square$}
\theoremstyle{plain}
\theorembodyfont{\upshape}
\newtheorem{remark}[theorem]{Remark}

\theoremsymbol{}
\theoremstyle{break}
\theorembodyfont{\slshape}

\theoremheaderfont{\scshape}
\theorembodyfont{\upshape} 
\theoremstyle{nonumberplain}
\theoremseparator{} 
\theoremsymbol{\rule{1.3ex}{1.3ex}} 
\newtheorem{proof}{Proof}
\usepackage[square,numbers,sort,compress,merge]{natbib}
\hypersetup{
plainpages=false,
colorlinks=true,% à mettre avant la suite : supprime l'encadrement des liens
linkcolor=black, 
anchorcolor=black, 
citecolor=black, 
urlcolor=black, 
menucolor=black, 
filecolor=black, 
bookmarksopen=true,
bookmarksnumbered=true}

\numberwithin{equation}{section}

\hypersetup{
pdfauthor={Thierry Masson, Bruno Iochum},
pdftitle={???},
pdfsubject={},
pdfcreator={pdflatex},
pdfproducer={pdflatex},
pdfkeywords={}
}

\newcommand\R{{\mathbb R}}
\newcommand\C{{\mathbb C}}
\newcommand\N{{\mathbb N}}

\newcommand{\A}{\mathcal{A}}
\newcommand{\B}{\mathcal{B}}
\newcommand{\BB}{\mathbb{B}}

\newcommand{\DD}{\mathcal{D}}
\renewcommand{\H}{\mathcal{H}}
\newcommand{\hH}{\widehat{\mathcal{H}}}

\newcommand{\vc}{\vcentcolon =}           %%  :=  in definition 
\newcommand{\cv}{=\vcentcolon }           %%  :=  in definition 
\newcommand{\bbbone}{{\text{\usefont{U}{dsss}{m}{n}\char49}}}
\newcommand{\vol}{{\text{vol}}}
\newcommand{\pmu}{\partial_\mu}            %% Partial derivative with respect to mu
\newcommand{\pnu}{\partial_\nu}              %% Partial derivative with respect to nu
\newcommand{\prho}{\partial_\rho}            %% Partial derivative with respect to rho
\newcommand{\psigma}{\partial_\sigma}            %% Partial derivative with respect to rho
\newcommand{\nmu}{\nabla_\mu}             %% Nabla with respect to mu
\newcommand{\nnu}{\nabla_\nu}                %% Nabla with respect to nu

\newcommand{\tocc}{\xrightarrow{\text{c.c.}}} % right arrow "change of coordinates"
\newcommand{\togt}{\xrightarrow{\text{g.t.}}} % right arrow "gauge transformation"

\newcommand{\hnabla}{\widehat{\nabla}}

\DeclareMathOperator{\Tr}{Tr}	             %% trace
\DeclareMathOperator{\tr}{tr}	             %% trace
\DeclareMathOperator{\Ker}{Ker}	    %% tkernel
\DeclareMathOperator{\End}{End}	    %% Endomorphism
\DeclarePairedDelimiter\abs{\lvert}{\rvert}
\DeclarePairedDelimiter\norm{\lVert}{\rVert}

\newcounter{mnotecount}[section]
\renewcommand{\themnotecount}{\thesection.\arabic{mnotecount}}
\newcommand{\mnote}[1]%{}
{\protect{\stepcounter{mnotecount}}$^{\mbox{\footnotesize
$%\!\!\!\!\!\!\,
\bullet$\themnotecount}}$ \marginpar{%\color{red}%
\raggedright\tiny\em
$\!\!\!\!\!\!\,\bullet$\themnotecount: #1} }

\begin{document}

\title{Heat trace for Laplace type operators\\ with non-scalar symbols}
\author{B. Iochum, T. Masson}
\date{}
\maketitle

\begin{center}
Centre de Physique Théorique \\
{\small Aix Marseille Univ, Université de Toulon, CNRS, CPT, Marseille, France}
\end{center}
\vspace{1cm}
\begin{abstract}
For an elliptic selfadjoint operator $P =-[u^{\mu\nu}\partial_\mu \partial_\nu +v^\nu \partial_\nu +w]$ acting on a fiber bundle over a Riemannian manifold, where $u^{\mu\nu},v^\mu,w$ are $N\times N$-matrices, we develop a method to compute the heat-trace coefficients $a_r$ which allows to get them by a pure computational machinery. It is exemplified in dimension 4 by the value of $a_1$ written both in terms of $u^{\mu\nu}=g^{\mu\nu}u,v^\mu,w$ or diffeomorphic and gauge invariants. We also answer the question: when is it possible to get explicit formulae for $a_r$?
\end{abstract}

\vspace{0.5cm}

\noindent {\it Keywords}: Heat kernel; non minimal operator; asymptotic heat trace; Laplace type operator

%%%%%%%%%%%%%%%%%%%%%%%%%%%
\section{Introduction}
%%%%%%%%%%%%%%%%%%%%%%%%%%%

We consider a compact Riemannian manifold $(M,g)$ without boundary and of dimension $d$ together with the nonminimal differential operator 
\begin{align}
\label{def-P}
P \vc-[u^{\mu\nu}(x)\pmu\pnu +v^\nu(x)\pnu +w(x)].
\end{align}
which is a differential operator on a smooth vector bundle $V$ over $M$ of fiber $\C^N$ where $u^{\mu\nu},\, v^\nu,\, w$ are $N\times N$-matrices valued functions.  This bundle is endowed with a hermitean metric. We work in a local trivialization of $V$ over an open subset of $M$ which is also a chart on $M$ with coordinates $(x^\mu)$. In this trivialization, the adjoint for the hermitean metric corresponds to the adjoint of matrices and the trace on endomorphisms on $V$ becomes the usual trace $\tr$ on matrices. 
Since we want $P$ to be a selfadjoint and elliptic operator on $L^2(M,V)$, we first assume that $u^{\mu\nu}(x)\,\xi_\mu\xi_\nu$ is a positive definite matrix in $M_N$:
\begin{align}
\label{Hyp-defpositiv}
u^{\mu\nu}(x)\,\xi_\mu\xi_\nu \text{ has only strictly positive eigenvalues for any $\xi \neq 0$.}
\end{align}

We may assume without loss of generality that $u^{\mu \nu}=u^{\nu\mu}$. In particular $u^{\mu\mu}$ is a positive matrix for each $\mu$ and each $u^{\mu\nu}$ is selfadjoint. \\
The asymptotics of the heat-trace
\begin{align}
\label{heat-trace-asympt}
\Tr e^{-tP} \underset{t \downarrow 0^+}{\sim} \,\sum_{r=0}^\infty a_r(P)\,t^{r-d/2}
\end{align}
exists by standard techniques (see \cite{Gilkeybook}), so we want to compute these coefficients $a_r(P)$.\\
While the spectrum of $P$ is a priori inaccessible, the computation of few coefficients of this asymptotics is eventually possible. The related physical context is quite large: the operators $P$ appeared in gauge field theories, string theory or the so-called non-commutative gravity theory (see for instance the references quoted in \cite{Avramidi2004, Avramidi2006, Avramidibook}). The knowledge of the coefficients $a_r$ are important in physics. For instance, the one-loop renormalization in dimension four requires $a_1$ and $a_2$.

When the principal symbol of $P$ is scalar ($u^{\mu\nu}=g^{\mu\nu} \bbbone_N$), there are essentially two main roads towards the calculation of the heat coefficients (with numerous variants): the first is analytical and based on elliptic pseudodifferential operators while the second is more concerned by the geometry of the Riemannian manifold $M$ itself with the search for invariants or conformal covariance. Compared with the flourishing literature existing when the principal symbol is scalar, there are only few works when it is not. One can quote for instance the case of operators acting on differential forms \cite{GBF, BGP,AV}. The first general results are in \cite{AB} or in the context of spin geometry using the Dirac operators or Stein-Weiss operators \cite{AB1} also motivated by physics \cite{Avramidi2004}. See also the approach in \cite{Fulling,GG,GK,GGR,Korniak, Anan, Guen, MT,Toms}.\\
The present work has a natural algebraic flavor which is the framework of operators on Hilbert space, thus with standard analytical part, so is related with the first road. In particular, it gives all ingredients to produce mechanically the heat coefficients. It is also inspired by the geometry {\it \`a la} Connes where $P=\DD^2$ for a commutative spectral triples $(\A,\H,\DD)$, thus has a deep motivation for noncommutative geometry.

Let us now enter into few technical difficulties.\\
While the formula for $a_0(P)$ is easily obtained, the computation of $a_1(P)$ is much more involved. To locate some difficulties, we first recall the parametrix approach, namely the use of momentum space coordinates $(x,\xi)\in T_x^*M$:
 \begin{align*}
& d_2(x,\xi) = u^{\mu\nu}(x)\,\xi_\mu\xi_\nu\,,\\
& d_1(x,\xi) = -iv^{\mu}(x)\,\xi_\mu\,, \\
& d_0(x)=-w(x).
\end{align*}
Then we can try to use the generic formula (see \cite{GBF})
\begin{align}
\label{eq-coeff}
a_r(P)=\tfrac{1}{(2\pi)^d} \tfrac{1}{-i2\pi}\int  dx\, d \lambda\, d\xi 
\,e^{-\lambda} \tr\,[b_{2r}(x,\xi,\lambda)]\,
\end{align}
where $\lambda$ belongs to a anticlockwise curve $\mathcal{C}$ around $\R^+$ and $(x,\,\xi) \in T^*(M)$. Here the functions $b_{2r}$ are defined recursively by
\begin{align*}
& b_0(x,\xi,\lambda) \vc(d_2(x,\xi)-\lambda)^{-1},\\
&b_r(x,\xi,\lambda) \vc -\sum_{\substack{r=j+\abs{\alpha}+2-k\\j<r}} \tfrac{(-i)^{\abs{\alpha}}}{\alpha!} \, (\partial_\xi^\alpha b_j)(\partial_x^\alpha d_k) \,b_0.
\end{align*}
The functions $b_{2r}$, even for $r=1$, generate typically terms of the form
\begin{align*}
\tr [A_1(\lambda)B_1A_2(\lambda)B_2 A_3(\lambda)\cdots]
\end{align*}
where all matrices $A_i(\lambda)=(d_2(x,\xi)-\lambda)^{-n_i}$ commute but do not commute a priori with $B_i$, so that the integral in $\lambda$ is quite difficult to evaluate in an efficient way. Of course, one can use the spectral decomposition $d_2 = \sum_i\lambda_i \,\pi_i$ to get, 
\begin{align}
\label{integration spectrale}
\sum_{i_1,i_2,i_3,\dots} \big[\int_{\lambda \in \mathcal{C}} d\lambda\, e^{-\lambda}\,(\lambda_{i_1}-\lambda)^{-n_{i_1}}(\lambda_{i_2}-\lambda)^{-n_{i_2}}(\lambda_{i_3}-\lambda)^{-n_{i_3}}\cdots \big]\tr (\pi_{i_1}B_1\pi_{i_2} B_2 \pi_{i_3}\cdots ).
\end{align}
While the $\lambda$-integral is easy via residue calculus, the difficulty is then to recombine the sum.
\\ This approach is conceptually based on the control between the resolvent $(P-\lambda)^{-1}$ which is not a pseudodifferential operator and the parametrix of the differential operator $(P-\lambda)$. 

Because of previous difficulties, we are going to follow another strategy, using a purely functional approach for the kernel of $e^{-tP}$  which is based on the Volterra series (see \cite{BGV,Avramidibook}). This approach is not new and has been used for the same purpose in \cite{AB,Avramidi2004,Avramidi2006}. \\
However our strategy is more algebraic and more in the spirit of rearrangement lemmas worked out in \cite{CM2014,Lesch}. In particular we do not go through the spectral decomposition of $u^{\mu\nu}$ crucially used in \cite{AB} (although in a slightly more general case than the one of Section~\ref{Section-example}). 
To explain this strategy, we need first to fix few notations.

Let $K(t,x,x')$ be the kernel of $e^{-tP}$ where $P$ is as in \eqref{def-P} and satisfies \eqref{Hyp-defpositiv}. Then
\begin{align*}
&\Tr[e^{-tP}\,] = \int dx\, \tr[K(t,x,x)],\\
& K(t,x,x) = \tfrac{1}{(2\pi)^d}\int d\xi \,e^{-ix.\xi} \,(e^{-tP} \,e^{i x.\xi}).
\end{align*}
When $f$ is a matrix-valued function on $M$, we get
\begin{align*}
-P (e^{ix.\xi}f)(x) &=\big(e^{ix.\xi}\,[-u^{\mu\nu}\xi_\mu\xi_\nu +2i u^{\mu\nu}\xi_\mu\pnu  +i v^\mu\xi_\mu + w(x)] f\big)(x)\\
& =-\big(e^{ix.\xi}\,[H+K +P] f\big)(x)
\end{align*}
where we used
\begin{align}
& H (x,\xi) \vc u^{\mu\nu}(x)\,\xi_\mu\xi_\nu \,,\label{def-H}\\
& K (x,\xi) \vc  -i\xi_\mu[v^\mu(x) +2u^{\mu\nu}(x) \,\pnu ]. \label{def-K}
\end{align}
Thus $H$ is the principal symbol of $P$ and it is non-scalar for non-trivial matrices $u^{\mu\nu}$.

If $\bbbone(x)=\bbbone$ is the unit matrix valued function, we get $e^{-tP} e^{ix.\xi}=e^{ix.\xi}\, e^{-t(H+K+P)} \,\bbbone$, so that, after the change of variables $\xi \to t^{1/2}\xi$,  the heat kernel can be rewritten as
\begin{align}
K(t,x,x) & = \tfrac{1}{(2\pi)^d}\int d\xi \,e^{-t(H+K+P)} \,\bbbone =\tfrac{1}{t^{d/2}}\tfrac{1}{(2\pi)^d}  \int d\xi \,e^{-H-\sqrt{t} K -tP}\,\bbbone. \label{Volterra}
\end{align}
A repetitive application of Duhamel formula  (or Lagrange's variation of constant formula) gives the Volterra series (also known to physicists as Born series): 
\begin{align*}
e^{A+B}=e^A +\sum_{k=1}^\infty \int_0^1 ds_1 \int_0^{s_1} ds_{2} \cdots\int_0^{s_{k-1}} ds_k\, \,e^{(1-s_1)A}\,B\,e^{(s_1-s_{2})A} \cdots e^{(s_{k-1}-s_k)A} \,B \,e^{s_k A}\,.
\end{align*}
Since this series does not necessarily converges for unbounded operators, we use only its first terms to generate the asymptotics \eqref{heat-trace-asympt} from \eqref{Volterra}. When $A=-H$ and $B=-\sqrt{t}K-tP$, it yields for the integrand of \eqref{Volterra}
\begin{align}
e^{-H-\sqrt{t} K -tP} \bbbone =\Big(&e^{-H} -\sqrt{t} \int_0^1 ds_1 \,e^{(s_1-1)H}\,K\,e^{-s_1 H} \nonumber\\
& +t\big[\int_0^1 ds_1 \int_0^{s_1} ds_2\, e^{(s_1-1)H}\,K\,e^{(s_2 -s_1)H}\, K\,e^{-s_2 H}-\int_0^1 ds_1\,e^{(s_1-1)H}\,P\,e^{-s_1 H} \big]\nonumber \\
& + \mathcal{O}(t^2)\Big) \bbbone. \label{Volt}
\end{align}
After integration in $\xi$, the term in $\sqrt{t}$ is zero since $K$ is linear in $\xi$ while $H$ is quadratic in $\xi$, so that 
\begin{align*}
\tr K(t,x,x)  \underset{t\downarrow 0}{\simeq} \tfrac{1}{t^{d/2}} [a_0(x)+t\,a_1(x)+\mathcal{O}(t^2)]
\end{align*}
with the local coefficients
\begin{align}
& a_0 (x) = \tr \tfrac{1}{(2 \pi)^{d}} \int d\xi\, e^{-H(x,\xi)}  \label{a0(x)} ,\\
& a_1(x)  =  \tr \tfrac{1}{(2 \pi)^{d}}\int d\xi\,\big[\int_0^1 ds_1 \int_0^{s_1} ds_2\, e^{(s_1-1)H}\,K\,e^{(s_2 -s_1)H}\, K\,e^{-s_2 H}\big]\nonumber\\
&\qquad\qquad - \tr \tfrac{1}{(2 \pi)^{d}}\int d\xi\,[\int_0^1 ds_1 \,e^{(s_1-1)H}\,P\, e^{-s_1 H}]\label{a1(x)}
\end{align}
where the function $\bbbone$ has been absorbed in the last $e^{-s_i H}$.\\
The coefficients $a_0(P)$ and $a_1(P)$ are obtained after an integration in $x$. Since we will not perform that integration which converges when manifold $M$ is compact, we restrict to $a_r(x)$.

We now briefly explain how we can compute $a_r(P)$. Expanding $K$ and $P$ in $a_r(x)$, one shows in Section~\ref{Section-algebraic} that all difficulties reduce to compute algebraic expressions like (modulo the trace)
\begin{align}
\label{Volterraterm}
\tfrac{1}{(2 \pi)^{d}}\int d\xi\,\,\,\int_0^1ds_1\int_0^{s_1}ds_2 \cdots \int_0^{s_{k-1}}ds_k\,\, e^{(s_1-1)H} \,B_1 \,e^{(s_2-s_1)H} \,B_2 \cdots B_k\, e^{-s_k H}
\end{align}
where the $B_i$ are $N\times N$-matrices equal to $u^{\mu\nu}$, $v^\mu$, $w$ or their derivatives of order two at most. Moreover, we see \eqref{Volterraterm} as an $M_N$-valued operator acting on the variables $(B_1,\dots,B_k)$ which precisely allows to focus on the integrations on $\xi$ and $s_i$ independently of these variables. Then we first compute the integration in $\xi$, followed by the iterated integrations in $s_i$. The main result of this section is \eqref{Tkp versus Ink} which represents the most general operator used in the computation of $a_r$. We will see that the previously mentioned integrations are manageable in Section~\ref{Section-Integral}. Actually, we show that we can reduce the computations to few universal integrals and count the exact number of them which are necessary to get $a_r(x)$ in arbitrary dimension. In Section \ref{Section-example}, we reduce to the case  $u^{\mu\nu}=g^{\mu\nu}u$ where $u$ is a positive matrix and explicitly compute the local coefficient $a_1$ for $\dim(M)=4$ in Theorem \ref{calcul de a1} in terms of $(u,\,v^\mu, \,w)$. Looking after geometric invariants like the scalar curvature of $M$, we swap the variables $(u,\,v^\mu,\,w)$ with some others based on a given connection $A$ on $V$. This allows to study the diffeomorphic invariance and covariance under the gauge group of the bundle $V$. The coefficient  $a_1$ can then be simply written in terms of a covariant derivative (combining $A$ and Christoffel symbols). In Section \ref{method}, we use our general results to address the following question: is it possible to get explicit formulae for $a_r$ avoiding the spectral decomposition like \eqref{integration spectrale}? We show that there is a positive answer only when $d$ is even and $r<d/2$ when $u^{\mu\nu}=g^{\mu\nu} u$ with non-trivial $u,\,v^\mu,\,w$. Finally, the case $u^{\mu\nu}=g^{\mu\nu}u + X^{\mu\nu}$ is considered as an extension of $u^{\mu\nu}=g^{\mu\nu}\bbbone + X^{\mu\nu}$ which appeared in the literature.

%%%%%%%%%%%%%%%%%%%%%%%%%%%%%%%%%
\section{\texorpdfstring{Formal computation of $a_r(P)$}{Formal computation of ar(P)}}
\label{Section-algebraic}
%%%%%%%%%%%%%%%%%%%%%%%%%%%%%%%%%

This section is devoted to show that the computation of $a_r(x)$ as \eqref{a1(x)} reduces to the one of terms like in \eqref{Volterraterm}. Since a point $x\in M$ is fixed here, we forget to mention it, but many of the structures below are implicitly defined as functions of $x$.

For $k\in \N$, let $\Delta_k$ be the $k$-simplex
\begin{align*}
& \Delta_k \vc \{s=(s_0,\cdots s_k)\in \R_+^{k+1} \, \vert \, 0 \leq s_k \leq s_{k-1} \leq \cdots \leq s_2 \leq s _1 \leq s_0=1\},\\
& \Delta_0 \vc \varnothing \,\,\text{by convention}.
\end{align*}
We use the algebra $M_N$ of $N \times N$-complex matrices. \\
Denote by $M_N[\xi,\partial]$ the complex vector space of polynomials both in $\xi=(\xi_\mu) \in \R^d$ and $\partial=(\pmu)$ which are $M_N$-valued differential operators and polynomial in $\xi$; for instance, $P,\,K, \, H \in M_N[\xi,\partial]$ with $P$ of order zero in $\xi$ and two in $\partial$, $K$ of order one in $\xi$ and $\partial$, and $H$ of order two in $\xi$ and zero in $\partial$.

For any $k\in \N$, define a map $f_k (\xi): M_N[\xi,\partial]^{\otimes^k} \to M_N[\xi,\partial]$, evidently related to \eqref{Volterraterm}, by
\begin{align}
 & f_k(\xi)[B_1 \otimes \cdots \otimes  B_k]\vc \int_{\Delta_k}ds\, e^{(s_1-1)H(\xi)} \,B_1 \,e^{(s_2-s_1)H(\xi)} \,B_2 \cdots B_k\, e^{-s_k H(\xi)},\label{fk avec arguments}\\
&  f_0(\xi)[a] \vc a\, e^{-H(\xi)}, \quad \text{for }a\in \C \cv M_N^{\otimes^0}.
\end{align}
Here, by convention, each $\pmu$ in $B_i\in M_N[\xi,\partial]$ acts on all its right remaining terms. Remark that the map $\xi \mapsto f_k(\xi)$ is even.\\
We first rewrite \eqref{Volt} in these notations (omitting the $\xi$-dependence):
\begin{align}
e^{-H-\sqrt{t} K -tP}  &=e^{-H}+\sum_{k=1}^\infty (-1)^k f_k[(\sqrt{t} K +t P)\otimes \cdots \otimes (\sqrt{t}K +t P)] \label{generateur de t}\\
&=e^{-H} \nonumber\\
&\qquad-t^{1/2} f_1[K] \nonumber \\
&\qquad+t(f_2[K\otimes K]-f_1[P] ) \nonumber\\
&\qquad+t^{3/2}(f_2[K\otimes P]+f_2[P\otimes K]-f_3[K\otimes K\otimes K])\nonumber\\
&\qquad+t^2(f_2[P\otimes P] -f_3[K\otimes K\otimes P]-f_3[K\otimes P \otimes K]-f_3[P \otimes K \otimes K])\nonumber\\
&\qquad + \mathcal{O}(t^2). \nonumber
\end{align}
Since all powers of $t$ in $(2n+1)/2$ have odd powers of $\xi_{\mu_1}\cdots \xi_{\mu_{p}}$ (with odd $p$), the $\xi$-integrals in \eqref{Volterraterm} will be zero since $f_k$ is even in $\xi$, so only
\begin{align}
& a_0(x)= \tr \tfrac{1}{(2 \pi)^{d}}\int d\xi\,f_0[1],\\
& a_1(x)= \tr \tfrac{1}{(2 \pi)^{d}}\int d\xi\,(f_2[K\otimes K]-f_1[P]), \label{eq-a1(x)}\\
& a_2(x)= \tr \tfrac{1}{(2 \pi)^{d}}\int d\xi\,(f_2[P\otimes P] -f_3[K\otimes K\otimes P]-f_3[K\otimes P \otimes K]-f_3[P \otimes K \otimes K]) \nonumber
\end{align} 
etc survive. \\
Our first (important) step is to erase the differential operator aspect of $K$ and $P$ as variables of $f_k$ to obtain variables in the space $M_N[\xi]$ of $M_N$-valued polynomials in $\xi$: because a $\partial$ contained in $B_i$ will apply on $e^{(s_{i+1}-s_i)H}B_{i+1}  \cdots  B_k e^{-s_kH}$, by a direct use of Leibniz rule and the fact that
\begin{align}
\label{derivee exp}
\partial e^{-sH}= -\int_0^s ds_1 e^{(s_1-s)H} \,(\partial H)\, e^{-s_1H},
\end{align}
we obtain the following

\begin{lemma}
\label{fk avec derivee}
When all $B_j$ are in $M_N[\xi, \partial]$, the functions $f_k$ for $k\in \N^*$ satisfy 
\begin{align}
f_k(\xi)[B_1 \otimes \cdots \otimes B_i\partial \otimes \cdots &\otimes B_k] =\sum_{j=i+1}^k f_k(\xi)[B_1 \otimes \cdots \otimes(\partial B_j)\otimes \cdots\otimes  B_k]\ \nonumber \\
&- \sum_{j=i}^k f_{k+1}(\xi)[B_1 \otimes \cdots \otimes B_j\otimes (\partial H)\otimes B_{j+1}\otimes \cdots\otimes  B_k]. \label{eq-fk avec derivee}
\end{align}
\end{lemma}

\begin{proof}
By definition (omitting the $\xi$-dependence)
\begin{align*}
f_k[B_1 \otimes \cdots \otimes B_i\partial \otimes \cdots\otimes  B_k] & \\
&\hspace{-2cm} =\int_{\Delta_k} ds\,e^{(s_1-1)H} B_1\, e^{(s_2-s_1)H}\,B_2 \cdots B_i \partial(e^{(s_{i+1}-s_i)H}B_{i+1} \cdots B_k \,e^{-s_kH}).
\end{align*}
The derivation $\partial$ acts on each factor in the parenthesis:

– On the argument $B_j$, $j\geq i+1$, which gives the first term of \eqref{eq-fk avec derivee}.

– On a factor $e^{(s_{j+1}-s_j)H}$ for $i\leq j \leq k-1$, we use \eqref{derivee exp}
\begin{align*}
\partial\, e^{(s_{j+1}-s_j)H}  =-\int_0^{s_j-s_{j+1}} ds' \, e^{s'+s_{j+1}-s_j)H} \,(\partial H)\,e^{-s'H}
 = -\int_{s_{j+1}}^{s_j} ds\, e^{(s-s_j)H} \,(\partial H)\, e^{(s_{j+1}-s)H} 
\end{align*}
with $s= s'+s_{j+1}$, so that in the integral, one obtains the term
\begin{align*}
- \int_0^1 ds_1\int_0^{s_1} ds_2 \cdots \int_0^{s_{k-1}}\int_{s_{j+1}}^{s_j} ds & \,e^{(s_1-1)H} B_1 e^{(s_2-s_1)H}B_2 \cdots \\
& \hspace{1cm}\cdots B_j e^{(s-s_j)H}(\partial H) \,e^{(s_{j+1}-s)H} B_{j+1} \cdots B_k\,e^{-s_kH}).
\end{align*}
Since, as directly checked, $\int_0^1ds_1\cdots \int_0^{s_{k-1}} ds_k \int_{s_j+1}^{s_j} ds = \int_{\Delta_{k+1}} ds'$ with $s'_j =s_j$ for $j\leq i-1$, $s'_i=s$ and $s'_j=s_{j-1}$ for $j\geq i+1$, this term is $-f_{k+1}[B_1 \otimes \cdots \otimes B_j\otimes (\partial H)\otimes B_{j+1}\otimes \cdots\otimes  B_k]$.

– Finally, on the factor $e^{-s_kH}$, one has $\partial e^{-s_k H}= - \int_0^{s_k} \,e^{(s-s_k)H} \,(\partial H) \, e^{-sH}$ which gives the last term: $-f_{k+1}[B_1 \otimes \cdots \otimes B_k \otimes (\partial H)]$.
\end{proof}
Thus \eqref{generateur de t} reduces to compute $f_k[B_1\otimes \cdots\otimes B_k]$ where $B_i\in M_N[\xi]$. \\
Our second step is now to take care of the $\xi$-dependence: by hypothesis, each $B_i$ in $B_1\otimes \cdots \otimes B_k$ has the form $\sum_{}B^{\mu_1\dots \mu_{\ell_i}}\,\xi_{\mu_1}\cdots \xi_{\mu_{\ell_i}}$ with $B^{\mu_1\dots \mu_{\ell_i}} \in M_N$, so that $B_1\otimes \cdots \otimes B_k$ is a sum of terms like $\BB_k^{\mu_1\dots \mu_{\ell}}\,\xi_{\mu_1}\cdots \xi_{\mu_{\ell}}$ where $\BB_k^{\mu_1\dots\mu_\ell}  \in M_N^{\otimes^k}$. As a consequence, by linearity of $f_k$ in each variable, computation of $a_r$ requires only to evaluate terms like
\begin{align}
\label{int of fk as operator}
\tfrac{1}{(2\pi)^d}\int d\xi \, \xi_{\mu_1}\cdots \xi_{\mu_\ell} \,f_k(\xi)[\,\BB_k^{\mu_1\dots\mu_\ell}] \in M_N \quad\text{ with }\BB_k^{\mu_1\dots\mu_\ell}  \in M_N^{\otimes^k},
\end{align}
and we may assume that $\ell=2p, \,p\in \N$.

The next step in our strategy is now to rewrite the $f_k$ appearing in \eqref{int of fk as operator} in a way which is independent of the variables $\BB_k^{\mu_1\dots\mu_\ell}$, a rewriting obtained in \eqref{fk}. Then the driving idea is to show firstly that such $f_k$ can be computed and secondly that its repeated action on all variables which pop up by linearity (repeat that $K$ has two terms while $P$ has three terms augmented by the action of derivatives, see for instance \eqref{a1(x)}) 
is therefore a direct computational machinery.\\ For such rewriting we need few definitions justified on the way. 

 For $k\in \N$, define the (finite-dimensional) Hilbert spaces
\begin{align*}
\H_k \vc M_N^{\otimes^ k}, \quad \H_0 \vc \C,
\end{align*}
endowed with the scalar product
\begin{align*}
\langle A_1 \otimes \cdots \otimes A_k,B_1\otimes \cdots \otimes B_k \rangle_{\H_k} \vc \tr(A_1^* B_1) \cdots \tr (A_k^*B_k), \quad \langle a_0,b_0\rangle_{\H_0} \vc \overline{a}_0b_0,
\end{align*}
so each $M_N$ is seen with its Hilbert--Schmidt norm and $\norm{A_1 \otimes \cdots \otimes A_k}^2_{\H_k}=\prod_{j=1}^k \tr(A_j^*A_j)$.

We look at \eqref{int of fk as operator} as the action of the operator $\tfrac{1}{(2\pi)^d}\int d\xi \, \xi_{\mu_1}\cdots \xi_{\mu_l} \,f_k(\xi)$ acting on the finite dimensional Hilbert space $\H_k$.
\\
 Denote by $\B(E,F)$ the set of bounded linear operators between the vector spaces $E$ and $F$ and let $\B(E) \vc \B(E,E)$. 
 For $k\in \N$, let
\begin{align*}
& \hH_k \vc  \H_{k+1},\text{ so }\,\hH_0=M_N\,,\\
& m: \hH_k \to M_N\,,&& m(B_0 \otimes \cdots \otimes  B_k) \vc B_0\cdots B_k\, \,\text{ (multiplication of matrices)}, \\
& \kappa : \H_k \to \hH_k\,, && \kappa(B_1 \otimes \cdots \otimes  B_k) \vc\bbbone \otimes B_1 \otimes \cdots \otimes  B_k\,,\\
& \iota: M_N^{\otimes^{k+1}}\to \B(\H_k,M_N), && \iota(A_0 \otimes \cdots \otimes  A_k)[B_1\otimes\cdots \otimes B_k]  \vc A_0 B_1 A_1 \cdots B_k A_k\,,\\
& \iota: M_N\to \B(\C,M_N), && \iota(A_0)[a]  \vc a A_0\,,\\
&\rho: M_N^{\otimes^{k+1}} \to \B(\hH_k), && \rho(A_0 \otimes \cdots \otimes  A_k)[B_0 \otimes \cdots \otimes B_k] \vc  B_0A_0 \otimes \cdots \otimes B_k A_k.
\end{align*}
For $A\in M_N$ and $k\in \N$, define the operators 
\begin{align*}
& R_i(A): \hH_k \to \hH_k \,\text{ for }i=0,\dots ,k\\
& R_i(A)[B_0 \otimes \cdots \otimes B_k] \vc B_0 \otimes \cdots \otimes B_i A\otimes \cdots \otimes B_k.
\end{align*}
Thus
\begin{align}
\label{produit tensoriel versus produit}
\rho(A_0 \otimes \cdots \otimes  A_k)=R_0(A)\cdots R_k(A).
\end{align}
As shown in Proposition \ref{prop-iso}, $\iota$ is an isomorphism.
The links between the three spaces $M_N^{\otimes ^{k+1}},\,\B(\hH_k)$ and $\B(\H_k,M_N)$ are summarized in the following commutative diagram  where $(m \,\circ\, \kappa^*)(C)[B_1\otimes \cdots \otimes B_k] = m (C[\bbbone \otimes B_1 \otimes \cdots \otimes  B_k])$:
\begin{equation}
\label{eq-diagram}
\vcenter{\xymatrix@R=10pt@C=45pt%
{
 & {\B(\hH_k)} \ar[dd]^{m \,\circ\, \kappa^*}
 \\
{M_N^{\otimes k+1}} \ar[ru]^{\rho}  \ar[rd]^{\simeq}_{\iota} & 
\\
& {\B(\H_k,M_N)}
}}
\end{equation}

For any matrix $A\in M_N$ and $s \in \Delta_k$, define
\begin{align*}
&c_k(s,A)  \vc (1-s_1)\,A\otimes \bbbone \otimes \cdots\otimes \bbbone \, + \, (s_1-s_2)\,\bbbone \otimes A \otimes \bbbone\otimes \cdots \otimes \bbbone
\\
 &\hspace{4cm} +\cdots +\, (s_{k-1}-s_k)\,\bbbone\otimes \cdots \otimes A \otimes \bbbone \, + \, s_k\,\bbbone \otimes \cdots \otimes \bbbone \otimes A,\\
 &c_0(s,A) \vc A
\end{align*}
where the tensor products have $k+1$ terms, so that $c_k(s,A) \in M_N^{\otimes^{k+1}}$.\\
This allows a compact notation since now 
\begin{align}
\label{fk}
f_k(\xi)= \int _{\Delta_k} ds\,\, \iota[e^{-\xi_\alpha\xi_\beta\, c_k(s,u^{\mu\nu})}] \in \B(\H_k,M_N), \quad \text{ with } f_0(\xi)=\iota(e^{-\xi_\alpha\xi_\beta u^{\alpha\beta}}),
\end{align}
and these integrals converge because the integrand is continuous and the domain $\Delta_k$ is compact. 
\\
Since we want to use operator algebra techniques, with the help of $c_k(s,A) \in M_N^{\otimes^{k+1}}$, it is useful to lift the computation of \eqref{fk} to the (finite dimensional $C^*$-algebra) $\B(\hH_k)$ as viewed in diagram \eqref{eq-diagram}. Thus, we define
\begin{align*}
& C_k(s,A) \vc \rho(c_k(s,A)) \in \B(\hH_k),
\end{align*}
and then, by \eqref{produit tensoriel versus produit}
$$
C_k(s,A)= (1-s_1)\,R_0(A)+(s_1-s_2)\,R_1(A)+ \cdots +(s_{k-1}-s_k)\,R_{k-1}(A)+ s_k\,R_k(A).
$$

\begin{remark}
\label{lift}
 All these distinctions between $\H_k$ and $\hH_k$ or $c_k(s,A)$ and $C_k(s,A)$ seems innocent so tedious. But we will see later on that the distinctions between the different spaces in \eqref{eq-diagram} play a conceptual role in the difficulty to compute the coefficients $a_r$. Essentially, the computations and results takes place in $\B(\hH_k)$ and not necessarily in the subspace $M_N^{\otimes k+1} \subset \B(\hH_k)$ (see \eqref{Tkp versus Ink} for instance).
\end{remark}
 
Given a diagonalizable matrix $A=C\,\text{diag}(\lambda_1,\dots,\lambda_n)\,C^{-1} \in M_N$, let $C^{ij} \vc CE^{ij}C^{-1}$ for $i,j=1,\dots, n$ where the $E^{ij}$ form the elementary basis of $M_N$ defined by $[E^{ij}]_{kl}\vc \delta_{ik} \, \delta_{jl}$.
\\
We have the easily proved result:
\begin{lemma}
\label{LandR}
We have
 
 i) $R_i(A_1A_2)=R_i(A_2)R_i(A_1)$ and $[R_i(A_1),\,R_j(A_2)]=0$ when $i\neq j$.

ii) $R_i(A)^*=R_i(A^*)$.

iii) When $A$ is diagonalizable, $A C^{ij}= \lambda_i \,C^{ij}$ and $C^{ij} A= \lambda_j \,C^{ij}$.\\
Thus, all operators $R_i(A)$ on $\hH_k$ for any $k\in \N$ have common eigenvectors
\begin{align*}
 R_i(A) [C^{i_0j_0}\otimes \cdots \otimes C^{i_kj_k}] = \lambda_{j_i} \,C^{i_0j_0}\otimes \cdots \otimes C^{i_kj_k},
\end{align*}
and same spectra as $A$. \\
In particular, there are strictly positive operators if $A$ is a strictly positive matrix.
 \end{lemma}
 
This means that $C_k(s,A) \geq 0$ if $A \geq 0$ and $s \in \Delta_k$, and this justifies the previous lift. Now, evaluating \eqref{fk} amounts to compute the following operators in $\B(\hH_k)$:
\begin{align}
\label{termgeneric}
& T_{k,p}(x)\vc \tfrac{1}{(2 \pi)^{d}} \int_{\Delta_k}ds \int d\xi\, \xi_{\mu_1}\cdots \xi_{\mu_{2p}} \, e^{-\xi_\alpha \xi_\beta \,C_k(s,u^{\alpha\beta}(x))}:\, \hH_k \to \hH_k, \quad p\in \N,\, k\in \N.\\
& T_{0,0}(x) := \tfrac{1}{(2 \pi)^{d}} \int d\xi\, \, e^{-\xi_\alpha \xi_\beta u^{\alpha\beta}(x)} \in M_N \simeq \rho(M_N)\subset \B(\hH_0),\label{term00}
\end{align}
where $T_{k,p}$ depends on $x$ through $u^{\alpha\beta}$ only. 
Their interest stems from the fact they are independent of arguments $B_0\otimes \cdots \otimes B_k \in \hH_k$ on which they are applied, so are the corner stones of this work. Using \eqref{eq-diagram}, the precise link between $T_{k,p}$ and $f_k(\xi)$ is
\begin{align*}
m\circ \kappa^*\circ T_{k,p}=\tfrac{1}{(2\pi)^d} \int d \xi\, \xi_{\mu_1}\cdots \xi_{\mu_{2p}}\, f_k(\xi).
\end{align*}
The fact that $T_{k,p}$ is a bounded operator is justified by the following
\begin{lemma}
The above integrals \eqref{termgeneric} and \eqref{term00} converges and $T_{k,p} \in \B(\hH_k)$.
\end{lemma}

\begin{proof}
We may assume $k\in \N^*$ since for $k=0$, same arguments apply. \\For any strictly positive matrix $A$ with minimal eigenvalues $\lambda_{\min}(A)>0$, Lemma \ref{LandR} shows that, for any $s \in \Delta_k$,
$$
C_k(s,A) \geq [(1-s_1) \lambda_{\min}(A)+(s_1-s_2) \lambda_{\min}(A)+\cdots +s_k\lambda_{\min}(A)] \,\bbbone_{\hH_k} = \lambda_{\min}(A) \,\bbbone_{\hH_k}.
$$
We claim that the map $\xi \in \R^d \mapsto \lambda_{min}( \xi_\alpha\xi_\beta u^{\alpha\beta})$ is continuous: the maps $\xi\in \R^d \mapsto \xi_\alpha\xi_\beta u^{\alpha\beta}$ and $0< a \in \B(\hH_k) \mapsto \inf (\text{spectrum}(a))=\norm{a^{-1}}$ are continuous (the set of invertible matrices is a Lie group). \\
We use spherical coordinates $\xi \in \R^d \to (\abs{\xi}, \sigma)\in \R_+ \times S^{d-1}$,
 where $\sigma\vc \abs{\xi}^{-1} \xi$ is in the Euclidean sphere $S^{d-1}$ endowed with its volume form $d\Omega$. \\
 Then $\lambda_{min}(\xi_\alpha \xi_\beta u^{\alpha\beta})=\abs{\xi}^2\, \lambda_{min}(\sigma_\alpha \sigma_\beta u^{\alpha\beta})>0$ (remark that $\sigma_\alpha \sigma_\beta u^{\alpha\beta}$ is a strictly positive matrix). Thus $c \vc\inf\{\lambda_{min}(\sigma_\alpha \sigma_\beta u^{\alpha\beta})\, \vert \, \sigma\in S^{d-1}\}>0$ by compactness of the sphere. The usual operator-norm of $\hH_k$ applied on the above integral $T_{k,p}\,$, satisfies
\begin{align*}
\norm{T_{k,p}} & \leq \int_{\Delta_k}ds\int_{\sigma \in S^{d-1}} d\Omega_g(\sigma) \int_0^\infty dr\,\norm{r^{d-1}r^{2p} \,\sigma_{\mu_1}\cdots \sigma_{\mu_{2p}}\,e^{-r^2 c\,\bbbone_{\hH_k}}}\\
&  \leq \int_{\Delta_k}ds\int_{\sigma\in S^{d-1}} d\Omega_g(\sigma) \int_0^\infty dr \, r^{d-1+2p} \,e^{-r^2 c}  =\vol(\Delta_k)\,\vol(S_g^{d-1})\, \tfrac{\Gamma(d/2+p)}{2} \,c^{-d/2-p}\,.\quad
\end{align*}
\end{proof}

For the $\xi$-integration of \eqref{termgeneric}, we use again spherical coordinates, but now $\xi=r \,\sigma$ with $r=(g^{\mu\nu}\xi_\mu\xi_\nu)^{1/2}$, $\sigma = r^{-1} \xi \in S^{d-1}_g$ (this sphere depends on $x \in M$ through $g(x)$) and define
\begin{align*}
u[\sigma] \vc u^{\mu\nu}\sigma_\mu\sigma_\nu
\end{align*}
which is a positive definite matrix for any $\sigma \in S_g^{d-1}$. Thus we get
\begin{align}
T_{k,p} & =\tfrac{1}{(2\pi)^d}\int_{\Delta_k}ds\,\int_{S^{d-1}_g} d\Omega_g(\sigma)\,\sigma_{\mu_1}\cdots \sigma_{\mu_{2p}}\int_0^\infty dr\, r^{d-1+2p}\,e^{-r^2 C_k(s,u[\sigma])} \label{Tkp}\\
& = \tfrac{\Gamma(d/2+p)}{2(2\pi)^d} \int_{S^{d-1}_g} d\Omega_g(\sigma)\,\sigma_{\mu_1}\cdots \sigma_{\mu_{2p}} \,\int_{\Delta_k}ds\,\,C_k(s,u[\sigma])^{-(d/2+p)}.\nonumber
\end{align}
Thus, we have to compute the $s$-integration $\int_{\Delta_k} ds\, C_k(s,u[\sigma])^{-\alpha}$ for $\alpha \in \tfrac 12 \N^*$. We do that via functional calculus, using Lemma \ref{LandR} iii), by considering the following integrals
\begin{align}
&I_{\alpha,k}(r_0,r_1,\dots, r_k)  \vc \int_{\Delta_k} ds\, [(1-s_1)r_0+(s_1-s_2)r_1+\dots +s_k r_k]^{-\alpha} \nonumber\\
&\hspace{3.3cm}=\int_{\Delta_k} ds\, [r_0+s_1(r_1-r_0)+\dots +s_k (r_k-r_{k-1})]^{-\alpha} & \label{def Ink}\\
&I_{\alpha,0}(r_0) \vc r_0^{-\alpha}, \,\text{for } \alpha \neq 0, \label{In0}
\end{align}
where $0\neq r_i\in \R_+$ corresponds, in the functional calculus, to positive operator $R_i(u[\sigma])$.\\
Such integrals converge for any $\alpha \in \R$ and any $k\in \N^*$, even if it is applied above only to $\alpha=d/2+k-r \in \tfrac 12 \N$. Nevertheless for technical reasons explained below, it is best to define $I_{\alpha,k}$ for an arbitrary $\alpha\in \R$. \\
In short, the operator $T_{k,p}$ is nothing else than the operator in $\B(\hH_k)$
\begin{align}
\label{Tkp versus Ink}
T_{k,p}=\tfrac{\Gamma(d/2+p)}{2(2\pi)^d} \int_{S^{d-1}_g} d\Omega_g(\sigma)\,\sigma_{\mu_1}\cdots \sigma_{\mu_{2p}} \,I_{d/2+p,k} \big(R_0(u[\sigma]),R_1(u[\sigma]),\dots,R_k(u[\sigma])\big).
\end{align}
Remark that $T_{k,p}$ depends on $x$ via $u[\sigma]$ and the metric $g$. 

\begin{remark}We pause for a while to make a connection with the previous work \cite{AB}. There, the main hypothesis on the matrix $u^{\mu\nu}\xi_\mu\xi_\nu$ is that all its eigenvalues are positive multiples of $g^{\mu\nu}\xi_\mu\xi_\nu$ for any $\xi\neq 0$. Under this hypothesis, we can decompose spectrally $u[\sigma]=\sum_i \lambda_i \pi_i[\sigma]$ where the eigenprojections $\pi_i$ depends on $\sigma$ but not the associated eigenvalues $\lambda_i$. Then, operator functional calculus gives
\begin{align}
\label{I projectors}
I_{d/2+p,k}\big(R_0(u[\sigma]),\dots,R_k(u[\sigma])\big)=\sum_{i_0,\dots,i_k}I_{d/2+p,k}(\lambda_{i_0},\dots,\lambda_{i_k} )\,R_0(\pi_{i_0}[\sigma])\cdots R_k(\pi_{i_k}[\sigma])
\end{align}
and
\begin{align*}
T_{k,p}=\tfrac{\Gamma(d/2+p)}{2(2\pi)^d} \sum_{i_1,\dots,i_k}  \, I_{d/2+p,k} (\lambda_{i_0},\dots,\lambda_{i_k})\int_{S^{d-1}_g} d\Omega_g(\sigma)\,\sigma_{\mu_1}\cdots \sigma_{\mu_{2p}}\,R_0(\pi_{i_0} [\sigma])\cdots R_k(\pi_{i_k}[\sigma])
\end{align*}
where all $\pi_{i_0}[\sigma],\dots,\pi_{i_k}[\sigma]$ commute as operators in $\B(\hH_k)$. However, we do not try to pursue in that direction since it is not very explicit due to the difficult last integral on the sphere; also we remind that we already gave up  in the introduction the use of the eigenprojections for the same reason. Instead, we give for instance a complete computation of $a_1$ in Section \ref{Section-example} directly in terms of matrices $u^{\mu\nu},v^\mu,w$, avoiding this spectral splitting in the particular case where $u^{\mu\nu}=g^{\mu\nu} u$ for a positive matrix $u$.
\end{remark}

In conclusion, the above approach really reduces, as claimed in the introduction, the computation of all $a_r(x)$ for an arbitrary integer $r$ to the control of operators $T_{k,p}$ which encode all difficulties since, once known, their applications on an arbitrary variable $B_1\otimes \cdots\otimes B_k$ with $B_i\in M_n$ are purely mechanical. For instance in any dimension of $M$, the calculus of $a_1(x)$ needs only to know, $T_{1,0},\, T_{2,1},\,T_{3,2},$ $T_{4,3}$. 
More generally, we have the following

\begin{lemma}
\label{number to calculate}
For any dimension $d$ of the manifold $M$ and $x\in M$, given $r\in \N$, the computation of $a_r(P)$ needs exactly to know each of the $3r+1$ operators $T_{k,k-r}$ where $r \leq k \leq 4r$ or equivalently to know $I_{d/2,r}, \,I_{d/2+1,r+1},\dots,\,I_{d/2+3r,4r}$.
\end{lemma}

\begin{proof}
As seen in \eqref{generateur de t}, using the linearity of $f_k$ in each argument, we may assume that in $f_k[B_1\otimes \cdots \otimes B_k]$ each argument $B_i$ is equal to $K$ or $P$, so generates $t^{1/2}$ or $t$ in the asymptotic expansion. Let $n_K$ and $n_P$ the number of $K$ and $P$ for such $f_k$ involved in $a_r(P)$. Since $a_r(P)$ is the coefficient of $t^r$, we have $\tfrac 12n_K+n_P=r$ and $k\geq r$. In particular, $n_K$ much be even.

When $B_i=K= -i\xi_\mu[v^\mu(x) +2u^{\mu\nu}(x) \,\pnu]$, again by linearity, we may assume that the argument in $f(k,p) \vc f_k(\xi)[B_1(\xi)\otimes \cdots \otimes B_k(\xi)]$ is a polynomial of order $2p$ since odd order are cancelled out after the $\xi$-integration. In such $f(k,p)$, the number of $\xi$ (in the argument) is equal to $n_K$, so that $p=\tfrac 12 n_K$, and the number of derivations $\partial$ is $n_K+2n_P$. 
\\
We count now all $f(k,p)$ involved in the computation of $a_r(P)$. We initiate the process with $(k,p)=(n_K+n_P,\tfrac 12 n_K)$, so $k-p=r$ and after the successive propagation of $\partial$ as in Lemma \ref{fk avec derivee}, we end up with $(k',p')$ where $k'-p'=r$: in \eqref{eq-fk avec derivee}, $k \to k+1$ while $p\to p+1$ since $\partial H$ appears as a new argument. So $(k',p')=(k',k'-r)$ and the maximum of $k'$ is 
$2n_K+3n_P$. Here, $n_P = 0,\dots,r$ and $n_K=0,\dots ,2r$, thus the maximum is for $k'=4r$.\\
All $f(k,k-r)$ with $r\leq k \leq 4r$ will be necessary to compute $a_r(P)$: Let $k$ be such that $r\leq k\leq 3r$. Then a term $f(k,k-r)$ will be obtained by the use of Lemma \ref{fk avec derivee} applied on $f_r[u^{\mu_1\nu_1}\partial_{\mu_1}\partial_{\nu_1}\otimes \cdots \otimes u^{\mu_r\nu_r}\partial_{\mu_r}\partial_{\nu_r}]$ with an action of $k-r$ derivatives on the $e^{sH}$ and the reminder on the $B_i$. The same argument, applied to $f_{2r}[\xi_{\mu_1} u^{\mu_1\nu_1}\partial_{\nu_1} \otimes \cdots \otimes \xi_{\mu_{2r}} u^{\mu_{2r}\nu_{2r}}\partial_{\nu_{2r}}]$, also generates a term  $f(k,k-r)$ when $2r \leq k\leq 4r$.
\\
Finally, remark that we can swap the $\xi$-dependence of $f(k,p)$ into the definition \eqref{termgeneric} of $T_{k,p}$ to end up with integrals which are advantageously independent of arguments $B_i$.

The case $r=0$ is peculiar: since $k=0$ automatically, we have only to compute $T_{0,0}$ in \eqref{term00} which gives $a_0(x)$ by \eqref{a0(x)}.
\\
The link between the $T$'s and the $I$ is given in \eqref{Tkp versus Ink}.
\\
All previous reasoning is dimension-free.
\end{proof}

Of course, in an explicit computation of $a_r(x)$, each of these $3r+1$ operators $T_{k,k-r}$ can be used several times since applied on different arguments $B_1\otimes \cdots \otimes B_k$. The integral $I_{d/2+p,k}$ giving $T_k,p$ in \eqref{Tkp versus Ink} will be explicitly computed in Section~\ref{Section-Integral}.

We now list the terms of $a_1(x)$. Using the shortcuts 
\begin{align*}
\bar v \vc \xi_\mu v^\mu, \qquad \bar u^\nu \vc \xi_\mu u^{\mu\nu},
\end{align*}
starting from~\eqref{eq-a1(x)} and applying Lemma~\ref{fk avec derivee}, we get 
\begin{align}
& -f_1[P] =-f_2[u^{\mu\nu} \otimes \pmu\pnu H]+2 f_3[u^{\mu\nu} \otimes \pmu H\otimes \pnu H]
-f_2[v^\mu \otimes \pmu H] + f_1[w] , \label{f1[P]}\\
&  f_2[K\otimes K]  =-f_2[\bar v \otimes \bar v] +2 f_3[\bar v\otimes \bar u^\nu \otimes \pnu H]-2 f_2[\bar u^\mu \otimes \pmu \bar v] +2 f_3[\bar u^\mu \otimes \pmu H \otimes \bar v]   \nonumber\\
 & \hspace{2.4cm} +2 f_3[\bar u^\mu \otimes \bar v\otimes \pmu H] + 4 f_3[\bar u^\mu \otimes \pmu \bar u^\nu \otimes \pnu H]+4f_3[\bar u^\mu \otimes \bar u ^\nu \otimes \pmu\pnu H] \nonumber\\
& \hspace{2.4cm} -4f_4[\bar u^\mu \otimes \pmu H\otimes \bar u ^\nu \otimes \pnu H] -4 f_4[\bar u^\mu \otimes \bar u^\nu \otimes \pmu H \otimes \pnu H] \nonumber\\
& \hspace{2.4cm} -4f_4[\bar u^\mu \otimes \bar u^\nu \otimes \pnu H \otimes \pnu H] \label{f2[K tenseur K]}.
\end{align}
This represents 14 terms to compute for getting $a_1(x)$.

%%%%%%%%%%%%%%%%%%%%%%%%%%%%%%%%%%%
\section{\texorpdfstring{Integral computations of $I_{\alpha,k}$}{Integral computations of Ialpha,k}}
\label{Section-Integral}
%%%%%%%%%%%%%%%%%%%%%%%%%%%%%%%%%%%

We begin with few interesting remarks on $I_{\alpha,k}$ defined in \eqref{def Ink} and \eqref{In0}:
\begin{proposition}
\label{properties of Ink}
Properties of functions $I_{\alpha,k}$:

i) Recursive formula valid for $1\neq \alpha\in \R$ and $k\in \N^*$:
\begin{align}
\label{recursive}
I_{\alpha,k}(r_0,\dots,r_k) = \tfrac{1}{(\alpha-1)}(r_{k-1}-r_k)^{-1}[I_{\alpha-1,k-1}(r_0,\dots,r_{k-2},r_k)-I_{\alpha-1,k-1} (r_0,\dots,r_{k-1})].
\end{align}
(The abandoned case $I_{1,k}$ is computed in Proposition \ref{n<k}.)

ii) Symmetry with respect to last two variables:
\begin{align*}
I_{\alpha,k}(r_0,\dots ,r_{k-1},r_k)=I_{\alpha,k}(r_0,\dots,r_k,r_{k-1}).
\end{align*}

iii) Continuity:\\
The functions $I_{\alpha,k} : (\R_+^*)^{k+1} \to \R_+^* $ are continuous for all $\alpha\in \R$.
\end{proposition}

\begin{proof}
i) In $I_{\alpha,k}(r_0,\dots,r_k)=\int_0^1ds_1\int_0^{s_1} ds_2 \cdots \int_0^{s_{k-1}} ds_k\,[r_0+s_1(r_1-r_0)+\dots+s_k(r_k-r_{k-1})]^{-\alpha}$, the last integral is equal to 
\begin{align*}
\tfrac{1}{(\alpha-1)}(r_{k-1}-r_k)^{-1}\,[&(r_0+s_1(r_1-r_0)+\dots +s_{k-1}(r_{k-1}-r_{k-2})+ s_{k-1}(r_k-r_{k-1}))^{-(\alpha-1)} \\
&-(r_0+s_1(r_1-r_0)+\dots +s_{k-1}(r_{k-1}-r_{k-2}))^{-(\alpha-1)} ]
\end{align*}
which gives the claimed relation. One checks directly from the definition \eqref{def Ink} of $I_{\alpha+1,1}(r_0,r_1)$ that \eqref{recursive} is satisfied for the given $I_{\alpha,0}$.

ii) $I_{\alpha,1}(r_0,r_1)=\int_0^1ds\, r_0+s(r_1-r_0)]^{-\alpha}=\int_0^1 ds' \,[r_1+s'(r_0-r_1)]^{-\alpha} = I_{\alpha,1}(r_1,r_0)$ after the change of variable $s \to s'=1-s$. 
The symmetry follows now using \eqref{recursive} for a recurrence process.

iii) The map $g(s, \bar{r}) \vc [(1-s_1)r_0+(s_1-s_2)r_1+\dots +s_k r_k]^{-\alpha} >0$ is continuous at the point  $\bar{r}\in (\R_+^*)^{k+1}$ and uniformly bounded in a small ball $\B$ around $\bar{r}$ since then, $g(s,\bar{r}) \leq \max(r_{min}^{-\abs{\alpha}},r_{max}^{\abs{\alpha}})$ where $r_{min \text{ or } max} \vc \text{min or max} \{r_i \, \vert \, \bar{r} \in \B\}>0$. Thus, since the integration domain $\Delta_k$ is compact, by Lebesgue's dominated convergence theorem, we get the continuity of $I_{\alpha,k}$.
\end{proof}

From Lemma~\ref{number to calculate} and \eqref{Tkp versus Ink}, the computation of $a_r(P)$, for $r \geq 1$ (since $a_0$ is known already by \eqref{a_0(x) cas gmunu}) in spatial dimension $d \geq 1$, requires all $I_{d/2+k-r,k}$ for $r\leq k\leq 4r$. The sequence $I_{d/2,r}, I_{d/2+1,r+1}, \dots, I_{d/2+3r,4r}$ belongs to the same recursive relation~\eqref{recursive}, except if there is a $s \in \N$ such that $d/2 + s = 1$, which can only happen with $d=2$ and $s=0$ (see Case~1 below). The computation of this sequence requires then to compute $I_{d/2,r}$ as the root of~\eqref{recursive}.

The function $I_{d/2,r}$ can itself be computed using \eqref{recursive} when this relation is relevant. \\
Case 1: $d$ is even and $d/2 \leq r$, recursive sequence \eqref{recursive} fails at $I_{1, r-d/2+1}$:
\begin{equation}
I_{1, r-d/2+1} \to I_{2, r-d/2+2} \to \cdots \to \underbrace{I_{d/2,r} \to I_{d/2+1,r+1} \to \cdots \to I_{d/2+3r,4r}}_{\text{used to compute $a_r(x)$}}. \label{chain 1}
\end{equation}
Case 2: $d$ is even and $r < d/2$, relation \eqref{recursive} never fails and
\begin{equation}
I_{d/2-r,0} \to I_{d/2-r+1,1} \to I_{d/2-r+2,2} \to \cdots \to \underbrace{I_{d/2,r} \to I_{d/2+1,r+1} \to \cdots \to I_{d/2+3r,4r}}_{\text{used to compute $a_r(x)$}}. \label{chain 2}
\end{equation}
Case 3: $d$ is odd, relation \eqref{recursive} never fails and
\begin{equation}
I_{d/2-r,0} \to I_{d/2-r+1,1} \to I_{d/2-r+2,2} \to \cdots \to \underbrace{I_{d/2,r} \to I_{d/2+1,r+1} \to \cdots \to I_{d/2+3r,4r}}_{\text{used to compute $a_r(x)$}}. \label{chain 3}
\end{equation}
In the latter case, the root is $I_{\alpha, 0}$ with $\alpha = d/2-r$ half-integer, positive or negative and both situation have to be considered separately.

The recursive relation \eqref{recursive}, which for $I_{\alpha,k}$ follows from the integration on the $k$-simplex $\Delta_k$, has a generic solution:

\begin{proposition}
\label{solution generale} Given $\alpha_0\in \R$, $k_0\in \N$ and a function $F: \R_+^* \to  \R_+^*$, let the function $J_{\alpha_0 +k_0,k_0}: (\R_+^*)^{k+1} \to \R^*_+$ be defined by
\begin{align*}
 J_{\alpha_0 +k_0,k_0} (r_0,\dots ,r_k) \vc c_{\alpha_0+k_0 ,k_0} \,\sum_{i=0}^{k_0} \,\big[\prod^{k_0}_{\substack{j=0 \\j\neq i}}\, (r_i-r_j)^{-1} \big] F(r_i).
\end{align*}

i) Ascending chain: Then, all functions $J_{\alpha_0+k,k}$ obtained by applying the recurrence formula \eqref{recursive} for any $k\in \N, \,k \geq k_0$ have the same form: 
\begin{align}
\label{Jalphak3}
J_{\alpha_0 +k,k} (r_0,\dots ,r_k) \vc c_{\alpha_0+k ,k} \,\sum_{i=0}^k \,\big[\prod^k_{\substack{j=0 \\j\neq i}}\, (r_i-r_j)^{-1} \big] F(r_i)
\end{align}
with
\begin{equation}
\label{calpha}
c_{\alpha_0+k,k}= \tfrac{(-1)^{k-k_0}}{(\alpha_0 +k_0)\cdots(\alpha_0+k-1)}\, c_{\alpha_0+k_0,k_0}\,\text{ for } k>k_0\,.
\end{equation}

ii) Descending chain: when $\alpha_0 \in \R \backslash \{-\N\}$, the functions $J_{\alpha_0+k,k}$ defined by \eqref{Jalphak3} for $k\in \N^*$ starting with $k_0=0$, the root $F(r_0)=J_{\alpha_0,0}(r_0)$ and 
\begin{equation}
\label{calpha1}
c_{\alpha_0+k,k}= \tfrac{(-1)^k}{\alpha_0 (\alpha_0+1)\cdots(\alpha_0+k-1)}
\end{equation}
satisfy \eqref{recursive}.
\end{proposition}

\begin{proof}
i) It is sufficient to show that
\begin{align*}
X\vc\tfrac{1}{\alpha-1}(r_{\ell-1}-r_\ell)^{-1} [ J_{\alpha-1,\ell-1}(r_0,\dots,r_{\ell-2},r_\ell)-J_{\alpha-1,\ell-1}(r_0,\dots,r_{\ell-1})].
\end{align*}
has precisely the form \eqref{Jalphak3} for $\ell=k_0+1$ and $\alpha=\alpha_0+k_0+1$.
We have
\begin{multline}
X= \tfrac{c_{\alpha-1,\ell-1}}{\alpha-1}(r_{\ell-1}-r_\ell)^{-1}\Big[\,
\sum_{i=0}^{\ell-2}\, \big( \prod^{\ell-2}_{\substack{j=0\\j\neq i }}(r_i - r_j)^{-1} \big) (r_i - r_\ell)^{-1} \, F(r_i) 
+ \prod^{\ell-2}_{j=0}(r_\ell-r_i)^{-1} \,F(r_\ell) 
\\ 
-\sum_{i=0}^{\ell-2}\, \big( \prod^{\ell-2}_{\substack{j=0\\j\neq i }}(r_i - r_j)^{-1} \big) (r_i - r_{\ell-1})^{-1} \, F(r_i) 
- \prod^{\ell-2}_{j=0}(r_{\ell-1}-r_i)^{-1} \,F(r_{\ell-1}) \,
\Big]. \label{X}
\end{multline}
We can combine the two sums on $i=0, \dots, \ell-2$ as:
\begin{equation*}
\sum_{j=0}^{\ell-2} \big(\prod^{\ell-2}_{\substack{j=0\\ j \neq i}}  (r_i-r_j)^{-1} \big) 
[(r_i-r_\ell)^{-1} -(r_i-r_{\ell-1})^{-1}] \,F(r_i) 
=
(r_\ell-r_{\ell-1}) \sum_{j=0}^{\ell-2}\, \big( \prod^{\ell}_{\substack{j=0\\ j \neq i}}\,  (r_i-r_j)^{-1} \big) \, F(r_i).
\end{equation*}
Including $(r_{\ell-1}-r_\ell)^{-1}$, the others terms in \eqref{X} corresponds to $[ \prod^{\ell}_{\substack{j=0\\ j \neq i}}\,  (r_i-r_j)^{-1} ] \, F(r_i)$ for $i = \ell-1$ and $i=\ell$ (up to a sign), so that
\begin{equation*}
X= - \tfrac{c_{\alpha-1,\ell-1}}{\alpha-1} \,\sum_{j=0}^{\ell}\, \big( \prod^{\ell}_{\substack{j=0\\ j \neq i}}\,  (r_i-r_j)^{-1} \big) \, F(r_i)
\end{equation*}
which yields $c_{\alpha_0+k_0+1,k_0+1}=-\tfrac{1}{\alpha_0+k_0}\,c_{\alpha_0+k_0,k_0}$ and so the claim \eqref{calpha}.

ii) It is the same argument with $k_0=0$ and the hypothesis $\alpha_0 \notin -\N$  guarantees the existence of \eqref{calpha1} and moreover $c_{\alpha_0,0}=1$.
\end{proof}

Proposition~\ref{solution generale} exhibits the general solution of Cases 2 and 3 in \eqref{chain 2} and \eqref{chain 3} (with $\alpha_0 =d/2-r$, for $d$ even and $\alpha_0>0$, or for $d$ odd), with, for both, $F(r_0) = I_{d/2-r,0}(r_0)= r_0^{-\alpha_0}$, so that
\begin{align}
I_{d/2-r+k,k}(r_0,\dots,r_k) =\tfrac{(-1)^k}{(d/2-r)(d/2-r+1)\cdots(d/2-r+k-1)} \,\sum_{i=0}^k \,\big[\prod^k_{\substack{j=0 \\j\neq i}}\, (r_i-r_j)^{-1} \big] r_i^{-(d/2-r)}. \label{I cases 2 3}
\end{align}
To control the reminder Case 1 of chain \eqref{chain 1} where $\alpha_0\in -\N$ (so for $d$ even and $\alpha_0 =d/2-r \leq 0$), we need to compute the functions $I_{1,k}$ for $k=r-d/2+1$. This is done below and shows surprisingly enough that the generic solution of Proposition~\ref{solution generale} holds true also for a different function $F(r_0)$.

\begin{proposition}
\label{n<k}
Case 1: $d$ even and $d/2\leq r$. \\
For any $n \in \N$ and $k \in \N^*$,
\begin{align}
I_{n+1,k+n}(r_0,\dots,r_{n+k})= \tfrac{(-1)^n}{n!\,(k-1)!}\,\sum_{i=0}^{n+k} \,\,\,\prod_{\substack{j=0\\ j\neq i}}^{n+k}\, \,(r_i-r_j)^{-1} r_i^{k-1} \log r_i \label{I1k}.
\end{align}
Thus, the solution \eqref{I1k} written as $I_{d/2-r+k,k}$ corresponds to $J_{\alpha_0+k,k}$, with $\alpha_0=d/2-r \in -\N$, in \eqref{Jalphak3} with $F(r_0)=r_0^{-\alpha_0} \log r_0$ and $c_{1,1-\alpha_0}=\tfrac{1}{(-\alpha_0)!}\,$.
\end{proposition}

\begin{proof}
We use the shortcuts
\begin{align*}
& y_{k-1} \vc r_0+s_1(r_1-r_0)+\dots +s_{k-1}(x_{k-1}-x_{k-2}),\qquad y_0 \vc r_0, \\
& h_n(x) \vc x^n \log x, \quad n\in \N.
\end{align*}
Thus, $ I_{1,k}(r_0,\dots,r_k)=\int_{\Delta_k} ds \, [y_{k-1}+s_k(r_k-r_{k-1})]^{-1}$ and 
\begin{align*}
& y_{k-1}=y_{k-2} +s_{k-1}(r_{k-1}-r_{k-2}),\\
& \int h_n(x) = \tfrac{1}{n+1}\, h_{n+1}(x) -\tfrac{1}{(n+1)^2} \,x^{n+1}.
\end{align*}
We claim that for any integer $l$ with $1\leq l<k$, we have
\begin{align}
\label{hyp-recurrence}
& I_{1,k}(r_0,\dots ,r_k)= c_l \int_{\Delta_{k-l}} ds \,\sum_{i=0}^l \prod^l_{\substack{j=0\\j\neq i}} (r_{k-i}-r_{k-j})^{-1}\,h_{l-1}[y_{k-l-1} +s_{k-l}(x_{k-i}-x_{k-l-1})], \\
& c_l=\tfrac{1}{(l-1)!} \,.
\end{align}
It holds true for $l=1$ since
\begin{align*}
I_{1,k}(r_0,\dots ,r_k) & =\int_{\Delta_{k-1}} ds \int_0^{s_{k-1}}ds_k\, [y_{k-1}+s_k(x_k-x_{k-1})]^{-1}\\
 & = \int_{\Delta_{k-1}} ds\, \tfrac{1}{r_k-r_{k-1}} \,[ h_0\big(y_{k-1}+s_{k-1}(r_k-r_{k-1})\big) -h_0(y_{k-1})]\\
& = \int_{\Delta_{k-1}} ds\, \tfrac{1}{r_k-r_{k-1}}\,h_0\big(y_{k-2}+s_{k-1}(r_{k}-r_{k-2})\big) \\
& \hspace{1.6cm}+ \tfrac{1}{r_{k}-r_{k-1}}\,h_0\big(y_{k-2}+s_{k-1}(r_{k-1}-r_{k-2})\big)
\end{align*}
with $c_1=1$.
Assuming \eqref{hyp-recurrence} valid for some $l$, we want to get it for $l+1$, so we need to compute the integral in $s_{k-l}$:
\begin{align*}
 c_l\,\sum_{i=0}^l \int_{\Delta_{k-l-1}} ds\,\int_0^{s_{k-l-1}} ds_{k-l}\,&\prod^l_{j=0,\,j\neq i} (r_{k-i}-r_{k-j})^{-1} h_{l-1}(y_{k-l-1}+s_{k-l}(r_{k-i}-r_{k-l-i})\\
& =c_l\,\sum_{i=0}^l \int_{\Delta_{k-l-1}} ds\, \prod^l_{j=0,\,j\neq i} (r_{k-i}-r_{k-j})^{-1} (r_{k-i}-r_{k-l-1})^{-1} \\& \hspace{2cm} \big[ l^{-1} \,h_l\big(y_{k-l-1}+s_{k-l-1}(r_k-r_{k-l-1})\big) \\
& \hspace{2cm} -l^{-1}\, h_l(y_{k-l-1})\\
& \hspace{2cm} -l^{-2}(y_{k-l-1}+s_{k-l-1}(r_{k-i}-r_{k-l-1})^l + l^{-2} y_{k-l-1}^l\big].
\end{align*}
The forth line term is 
\begin{align*}
-l^{-2} \sum_{i=0}^l \,\prod^l_{\substack{j=0\\ j\neq i}}(r_{k-i}-r_{k-j})^{-1} (r_{k-i}-r_{k-l-1})^{-1} s_{k-l-1}(r_{k-i}-r_{k-l-1})\,P(r_{k-i})
\end{align*}
where $P$ is a polynomial with degree $\text{deg}(P)<l$. But for $0\leq p < l$, 
$$
\sum_{i=0}^l \,\prod^l_{\substack{j=0\\ j\neq i}} (r_{k-i}-r_{k-j})^{-1} \,x_{k-i}^p=0
$$ 
by \eqref{eq-appendix 1}, so this term is null. The second line term is
\begin{align*}
\int_{\Delta_{k-l-1}} ds\, l^{-1}c_l \sum_{i=0}^l \,\prod^l_{\substack{j=0\\ j\neq i}} (r_{k-i}-r_{k-j})^{-1}\,h_l\big(y_{k-l-2}+s_{k-l-1}(r_{k-i}-r_{k-l-2})\big)
\end{align*}
while the third one is
\begin{align*}
\int_{\Delta_{k-l-1}} \hspace{-0.5cm}ds\, l^{-1}c_l \,h_l(y_{k-l-2}+s_{k-l-1}(r_{k-l-1}-r_{k-l-2}) \big[ \sum_{i=0}^l (r_{k-l-1}-r_{k-i})^{-1} \prod^l_{\substack{j=0\\ j\neq i}} (r_{k-i}-r_{k-j})^{-1} \big]
\end{align*}
where last bracket is equal to $\prod^{l+1}_{j=0,\,j\neq l+1} (r_{k-l-1}-r_{k-i})^{-1}$ by \eqref{eq-appendix 2}. Thus, it corresponds to $i=l+1$ in the sum, and these second and third lines give
\begin{align*}
l^{-1}c_l\int_{\Delta_{k-l-1}}  ds\, \prod^l_{j=0,\,j\neq i} (r_{k-i}-r_{k-j})^{-1} \,h_l\big(y_{k-l-2}+s_{k-l-1}(r_{k-i}-r_{k-l-2})\big)
\end{align*}
with $l^{-1} c_l=(l!)^{-1}$.\\
The computations made in previous recurrence can be adapted to the case $l=k$, starting from $l=k-1$ with same formulae using $y_0=r_0$ and $s_0=1$ which implement $y_0 +s_0(r_1-r_0)=r_1$, so that \eqref{hyp-recurrence} is exactly the formula
\begin{align}
\label{eq-I1k}
 I_{1,k}(r_0,\dots ,r_k)= \tfrac{1}{(k-1)!} \,\sum_{i=0}^k\, \prod^k_{j=0,\,j\neq i} (r_{i}-r_{j})^{-1}\,r_i^{k-1} \log r_i\,.
\end{align}
Remark that \eqref{eq-I1k} proves \eqref{I1k} for $n=0$ and we now apply Proposition \ref{solution generale} i) to show \eqref{I1k}.
\end{proof}

The next propositions compute explicitly Case 3 and then Case 2, and the result is not written as in \eqref{I cases 2 3} where denominators in $r_i - r_j$ appear. This allows to deduce algorithmically (\textsl{i.e.} without any integration) the sequence $I_{d/2,r}, I_{d/2+1,r+1}, \dots, I_{d/2+3r,4r}$.

\begin{proposition}
Case 3: $d$ is odd.\\
If $d/2-r =\ell+1/2$ with $\ell \in \N$, the root and its follower are 
\begin{align*}
& I_{\ell+1/2,0}(r_0) = r_0^{-\ell-1/2}, \\
& I_{\ell+3/2,1}(r_0,r_1) = \tfrac{2}{2\ell+1} (\sqrt{r_0}\sqrt{r_1}\,)^{-2\ell-1}(\sqrt{r_0}+\sqrt{r_1}\,)^{-1} \sum_{0\leq l_1\leq 2\ell} \sqrt{r_0}^{\,l_1}\,\sqrt{r_1}^{\,2\ell-l_1},
\end{align*}
while if $d/2-r=-\ell-1/2$ with $\ell\in \N$, the root and its follower are
\begin{align*}
& I_{-\ell-1/2,0}(r_0) = r_0^{\ell+1/2}, \\
& I_{-\ell+1/2,1}(r_0,r_1) = \tfrac{2}{2\ell+1} (\sqrt{r_0}+\sqrt{r_1}\,)^{-1} \sum_{0\leq l_1\leq 2\ell} \sqrt{r_0}^{\,l_1}\,\sqrt{r_1}^{\,2\ell-l_1}.
\end{align*}
\end{proposition}

\begin{proof}
Using \eqref{recursive} and \eqref{In0}, we get when $\ell\geq 0$
\begin{align*}
I_{\ell+3/2,1}(r_0,r_1) & = \tfrac{1}{\ell+1/2} (r_0-r_1)^{-1}[I_{\ell+1/2,0}(r_1)-I_{\ell+1/2,0}(r_0)]\\
& = \tfrac{2}{2\ell+1} (\sqrt{r_0}-\sqrt{r_1}\,)^{-1}(\sqrt{r_0}+\sqrt{r_1}\,)^{-1} [r_1^{-\ell-1/2} -r_0^{-\ell-1/2}]
\end{align*}
where the term in bracket is 
\begin{align*}
r_1^{-\ell-1/2} -r_0^{-\ell-1/2}&=(\sqrt{r_0}\sqrt{r_1}\,)^{-2\ell-1}[r_0^{2\ell+1} -r_1^{2\ell+1}]\\
& =(\sqrt{r_0}\sqrt{r_1}\,)^{-2\ell-1}(\sqrt{r_0}-\sqrt{r_1}\,)\sum_{0\leq l_1\leq 2\ell} \sqrt{r_0}^{\,l_1}\,\sqrt{r_1}^{\,2\ell-l_1},
\end{align*}
which gives the result. Similar proof for the other equality.
\end{proof}

This proposition exhibits only the two first terms of the recurrence chain in Case~3: similar formulae can be obtained at any level in which no $(r_i-r_j)^{-1}$ factors appear. Unfortunately, they are far more involved.

\begin{proposition}
\label{prop I case 2}
Case 2: $d$ even and $r<d/2$. \\
For $n\in \N$, $n\geq k+1$ and $k\in \N^*$,
\begin{align}
I_{n,k}&(r_0,\dots,r_k) \nonumber\\
&=\tfrac{(r_0\cdots r_k)^{-1}}{(n-1)\cdots (n-k)}\sum_{\substack{0\leq l_k\leq l_{k-1}\leq \cdots\\\cdots \leq l_1 \leq n-(k+1)}} r_0^{l_1-(n-(k+1))}\, r_1^{l_2-l_1}\cdots r_{k-1}^{l_k-l_{k-1}} \, r_k^{-l_k}   \label{Ink pour n>k neg} \\
&=\tfrac{(r_0\cdots r_k)^{-(n-k)}}{(n-1)\cdots (n-k)} \,\sum_{\substack{0\leq l_k\leq l_{k-1}\leq \cdots\\\cdots \leq l_1\leq n-(k+1)}}
r_0^{l_1}\,r_1^{l_2+(n-(k+1))-l_1} \cdots r_{k-1}^{l_k+(n-(k+1))-l_{k-1}}\, r_k^{(n-(k+1))-l_{k}}.  \label{Ink pour n>k pos}
\end{align}
In \eqref{Ink pour n>k pos}, all exponents in the sum are positive while they are negative in \eqref{Ink pour n>k neg}. 
In particular
\begin{align}
\label{I(n+k),k}
I_{n+k,k}(r_0,\dots,r_k)=\tfrac{(r_0\cdots r_k)^{-n}}{(n+k-1)\cdots (n+1) n} \,\sum_{\substack{0\leq l_k\leq l_{k-1}\leq  \cdots\\
\cdots \leq l_1\leq n-1}} r_0^{l_1}\,r_1^{l_2+(n-1)-l_1} \cdots r_{k-1}^{l_k+(n-1)-l_{k-1}}\, r_k^{(n-1)-l_{k}}. 
\end{align}
\end{proposition}

\begin{proof}
The first and second equalities follow directly from the third that we prove now. Equality \eqref{I(n+k),k} is true for $k=1$ (the case $k=0$ is just the convention \eqref{In0}) since
\begin{align*}
I_{n,1}(r_0,r_1) &= \int_0^1ds_1\, (r_0+s_1(r_1-r_0)^{-n}= \tfrac{1}{n-1} (r_0-r_1)^{-1}[r_1^{-n+1}-r_0^{-n+1}]\\
&=\tfrac{1}{n+1} (r_0r_1)^{-n+1}\sum_{l_1=0}^{n-2} r_0^{l_1} r_1^{n-2-l_1}.
\end{align*}
Assuming \eqref{I(n+k),k} holds true for $l=0,\dots, k-1$, formula \eqref{recursive} gives
\begin{align*}
I_{n+k,k}&(r_0,\dots,r_{k}) \\
& = \tfrac{1}{n+k-1}(r_{k-1}-r_k)^{-1} \big[ I_{n+k-1,k-1}(r_0,\dots,r_{k-2},r_k)-I_{n+k-1,k-1}(r_0,\dots,r_{k-2},r_{k-1}) \big].
\end{align*}
The term in bracket is
\begin{align*}
\tfrac{(r_0\cdots r_{k-2}r_k)^{-n}}{(n+k-2)\cdots n}  \sum_{\substack{0\leq l_{k-1}\leq l_{k-2}\leq \cdots\\\cdots \leq l_1\leq n-1}} r_0^{l_1} r_1^{l_2+(n-1)-l_1}\cdots r_{k-2}^{l_{k-1}+(n-1)-l_{k-2} }r_k^{n-1-l_{k-1}} \\
- \tfrac{(r_0\cdots r_{k-2} r_{k-1})^{-n}}{(n+k-2)\cdots n}
 \sum_{\substack{0\leq l_{k-1}\leq l_{k-2}\leq \cdots\\\cdots \leq l_1\leq n-1}} r_0^{l_1} r_1^{l_2+(n-1)-l_1}\cdots r_{k-2}^{l_{k-1}+(n-1)-l_{k-2}} r_{k-1}^{n-1-l_{k-1}}.
\end{align*}
Thus
\begin{align*}
I_{n+k,k}(r_0,\dots,r_{k+1})=\tfrac{(r_0\cdots r_{k-2})^{-n}}{(n+k-1)\cdots n} &\sum_{\substack{0\leq l_{k-1}\leq l_{k-2}\leq \cdots\\\cdots \leq l_1\leq n-1}} r_0^{l_1}r_1^{l_2 +(n-1)-l_1}\cdots r_{k-2}^{l_{k-1}+(n-1)-l_{k-2}} \\
&  \hspace{2cm}(r_{k-1}-r_k)^{-1}\big[ r_k^{-n}r_k^{n-1-l_{k-1}}-r_{k-1}^{-n} r_{k-1}^{(n-1)-l_{k-1}} \big].
\end{align*}
Since the last line is equal to
\begin{align*}
(r_{k-1} r_k)^{-1-l_{k-1}} \sum_{0 \leq l_k \leq l_{k-1}} r_{k-1}^{l_k} r_k^{l_{k-1}-l_k} = (r_{k-1} r_k)^{-n} \sum_{0 \leq l_k \leq l_{k-1}} r_{k-1}^{(n-1)+l_k-l_{k-1}} r_k^{(n-1)-l_k}
\end{align*}
we have proved \eqref{I(n+k),k}.
\end{proof}
The interest of \eqref{I(n+k),k} is the fact that in \eqref{Tkp versus Ink} we have the following: for $B_0\otimes\cdots\otimes B_k \in \hH_k$,
\begin{multline}
I_{n+k,k} \big(R_0(u[\sigma]),\dots ,R_k(u[\sigma])\big)[B_0\otimes\cdots\otimes B_k] \\
= \tfrac{1}{(n+k-1)\cdots (n+1)n} \,\sum_{\substack{0\leq l_k\leq l_{k-1}\leq  \cdots\\
\cdots \leq l_1\leq n-1}} B_0\,u[\sigma]^{l_1-n}\otimes B_1\,u[\sigma]^{l_2-l_1-1}\otimes \cdots \\
\cdots \otimes B_{k-1}\,u[\sigma]^{l_k-l_{k-1}-1}\otimes B_k \,u[\sigma]^{-l_{k}-1}. 
\label{I case 2 R(u)}
\end{multline}
Or viewed as an operator in $\B(\H_k,M_N)$ (see diagram \eqref{eq-diagram}):
\begin{multline*}
m \circ \kappa^*\circ I_{n+k,k}\big(R_0(u[\sigma]),\dots ,R_k(u[\sigma])\big)[B_1\otimes\cdots\otimes B_k] \\
=\tfrac{1}{(n+k-1)\cdots (n+1)n} \,\sum_{\substack{0\leq l_k\leq l_{k-1}\leq  \cdots\\
\cdots \leq l_1\leq n-1}} u[\sigma]^{l_1-n} B_1\,u[\sigma]^{l_2-l_1-1} B_2 \cdots  B_{k-1}\,u[\sigma]^{l_k-l_{k-1}-1} B_k \,u[\sigma]^{-l_{k}-1}. 
\end{multline*}
While, if one want to use directly \eqref{I cases 2 3} on $B_0\otimes\cdots\otimes B_k$, we face the difficulty to evaluate  $[R_i(u)-R_j(u)]^{-1} [B_0\otimes\cdots\otimes B_k]$ in $\hH_k$. 

Another defect of \eqref{I cases 2 3} shared by \eqref{I1k} is that it suggests an improper behaviour of integrals $I_{n+1,k+n}$ when two variables $r_i$ are equal. But the continuity proved in Proposition~\ref{properties of Ink} shows that this is just an artifact.

%%%%%%%%%%%%%%%%%%%%%%%%%%%%%%%%%%%%%%%
\section{\texorpdfstring{An example for $d=4$ and $u^{\mu\nu}=g^{\mu\nu} u$}{An example for d=4 and umunu=gmunu u}}
\label{Section-example}
%%%%%%%%%%%%%%%%%%%%%%%%%%%%%%%%%%%%%%%

Here, we explicitly compute $a_1(x)$ for $d=\dim(M)=4$ assuming $P$ satisfies \eqref{def-P} and \eqref{Hyp}.\\
Given a strictly positive matrix $u(x) \in M_N$ where $x \in (M,g)$, we satisfy Hypothesis \ref{Hyp-defpositiv} with
\begin{align}
\label{Hyp}
u^{\mu\nu}(x) \vc g^{\mu\nu}(x) \,u(x).
\end{align}
This implies that 
$$
H(x,\xi)=\abs{\xi}_{g(x)}^2 \, u(x) \text{ where }\abs{\xi}_{g(x)}^2 \vc g^{\mu\nu}(x)\, \xi_\mu \xi_\nu\,.
$$
Of course the fact that $u[\sigma]=u$ is then independent of $\sigma$, simplifies considerably \eqref{Tkp versus Ink} since the integral in $\xi$ can be performed. Thus we assume \eqref{Hyp} from now on and \eqref{Tkp versus Ink} becomes
\begin{align}
\label{Tkp example}
T_{k,p}=g_d\,G(g)_{\mu_1\dots \mu_{2p}}\, I_{d/2+p,k}\big(R_0(u),R_1(u),\dots,R_k(u)\big) \in \B(\hH_k)
\end{align}
with (see \cite{ZJ})
\begin{align}
& g_d \vc \tfrac{1}{(2\pi)^d} \int_{\R^d} d\xi\, e^{-\abs{\xi}_{g(x)}^2}= \tfrac{\sqrt{\abs{g}}}{2^{d}\,\pi^{d/2}}\,, \nonumber \\
& G(g)_{\mu_1\dots \mu_{2p}} \vc    \tfrac{1}{(2\pi)^d\,g_d}\int d\xi\, \xi_{\mu_1}\cdots \xi_{\mu_{2p}}\, e^{-g^{\alpha\beta} \xi_\alpha\xi_\beta} \nonumber\\
& \hspace{+2.05cm} =\ \tfrac{1}{2^{2p}\,p!} \, \big( \sum_{\rho \in S_{2p}} g_{\mu_{\rho(1)} \mu_{\rho(2)}}\cdots g_{\mu_{\rho(2p-1)} \mu_{\rho(2p)}} \big)  =\ \tfrac{(2p)!}{2^{2p}\,p!} \,g_{(\mu_1\mu_2\dots \mu_{2p})}
\label{G(g)}
\end{align}
where $\abs{g} \vc \det(g_{\mu\nu})$, $S_{2p}$ is the symmetric group of permutations on $2p$ elements and the parenthesis in the index of $g$ is the complete symmetrization over all indices. 

Using the shortcuts
$$
I_{d/2+p,k} \vc I_{d/2+p,k}\big(R_0(u),\dots,R_k(u)\big),
$$
the formula \eqref{int of fk as operator} becomes simply
\begin{align}
\label{terme de base}
\tfrac{1}{(2\pi)^d}\int d\xi \, \xi_{\mu_1}\cdots \xi_{\mu_{2p}} \,f_k(\xi)[\,\BB_k^{\mu_1\dots\mu_{2p}}]=g_d\, (m\circ I_{d/2+p,k})[\bbbone\otimes G(g)_{\mu_1\dots \mu_{2p}} \,\BB_k^{\mu_1\dots \mu_{2p}}].
\end{align}
In particular, it is possible to compute the dimension-free contractions $G(g)_{\mu_1\dots \mu_{2p}} \,\BB_k^{\mu_1\dots \mu_{2p}}$ before evaluating the result in the $I_{d/2+p,k}$'s.
For $a_0(x)$, we get
\begin{align}
\label{a_0(x) cas gmunu}
a_0(x)=\tr \tfrac{1}{(2\pi)^d} \int d\xi\, e^{-u(x)\abs{\xi}_{g(x)}^2}=g_d(x)\, \tr[u(x)^{-d/2}].
\end{align}
We divide the computation of $a_1(x)$ into several steps.

%%%%%%%%%%%%%%%%%%%%%
\subsection{Collecting all the arguments}
\label{subsection-collectcontraction}
%%%%%%%%%%%%%%%%%%%%%

As a first step, we begin to collect all terms $\BB_k^{\mu_1\dots \mu_{2p}}$ of \eqref{terme de base} due to the different variables appearing in \eqref{f1[P]} and \eqref{f2[K tenseur K]}, including their signs.

\underline{Variable in $f_1$}:  $w$.

\underline{Variables in $f_2$} without the common factor $\xi_{\mu_1}\xi_{\mu_2}$ and summation over $\mu_1, \mu_2$: 
\begin{align*}
\begin{array}{lcl}
-u^{\mu\nu} \otimes \pmu \pnu H &\to &
-g^{\mu\nu}(\pmu \pnu g^{\mu_1\mu_2})\,u\otimes u 
-2 g^{\mu\nu}(\pmu g^{\mu_1\mu_2}) \,u \otimes \pnu u 
- g^{\mu\nu} g^{\mu_1\mu_2}\, u \otimes \pmu \pnu u
\\[5pt]
-v^\mu \otimes \pmu H &\to &
-(\pmu g^{\mu_1\mu_2}) \,v^\mu \otimes u 
-g^{\mu_1\mu_2} \,v^\mu \otimes \pmu u 
\\[5pt]
-\bar v\otimes \bar v & \to &
-v^{\mu_1} \otimes v^{\mu_2}
\\[5pt]
-2 \bar u^\mu \otimes \pmu \bar v &\to &
-2 g^{\mu\mu_1} \, u\otimes \pmu v^{\mu_2}.
\end{array}
\end{align*}

\underline{Variables in $f_3$} without the commun factor $\Pi_{i=1}^4 \xi_{\mu_i}$ and summation over the $\mu_i$:
\begin{align*}\def\arraystretch{1.1}
\begin{array}{lcl}
2 u^{\mu\nu} \otimes \pmu H \otimes \pnu H &\to  &
+2 g^{\mu\nu}(\pmu g^{\mu_1\mu_2})(\pnu g^{\mu_3\mu_4}) \, u\otimes u\otimes u\\
&& 
+2g^{\mu\nu} (\pmu g^{\mu_1\mu_2})g^{\mu_3\mu_4}\, u\otimes u\otimes \pnu u\\
&& 
+2g^{\mu\nu}  g^{\mu_1\mu_2}(\pnu g^{\mu_3\mu_4})\, u\otimes \pmu u\otimes  u\\
&& 
+2g^{\mu\nu}  g^{\mu_1\mu_2} g^{\mu_3\mu_4}\, u\otimes \pmu u\otimes \pnu u 
\\[5pt]
2 \bar v \otimes \bar u^\mu \otimes \pmu H &\to  &
+2 g^{\mu\mu_2}(\pmu g^{\mu_3\mu_4})\,v^{\mu_1} \otimes u \otimes u 
+2 g^{\mu\mu_2}g^{\mu_3\mu_4}\,v^{\mu_1} \otimes u  \otimes \pmu u 
\\[5pt]
2 \bar u^\mu \otimes \pmu H \otimes \bar v &\to &
+2 g^{\mu\mu_1}(\pmu g^{\mu_2\mu_3})\, u\otimes u \otimes v^{\mu_4} 
+ 2 g^{\mu\mu_1}g^{\mu_2\mu_3}\, u\otimes \pmu u \otimes v^{\mu_4} 
\\[5pt]
2 \bar u^\mu \otimes \bar v \otimes \pmu H &\to &
+2 g^{\mu\mu_1}(\pmu g^{\mu_3\mu_4})\, u\otimes v^{\mu_2}\otimes u  
+ 2 g^{\mu\mu_1}g^{\mu_3\mu_4}\, u \otimes v^{\mu_2} \otimes \pmu u
\\[5pt]
4\bar u^\mu \otimes \pmu \bar u^\nu \otimes \pnu H &\to & 
+4g^{\mu\mu_1}(\pmu g^{\nu\mu_2}) (\pnu g^{\mu_3\mu_4}) \,u \otimes u \otimes u \\
&& 
+4g^{\mu\mu_1}(\pmu g^{\nu\mu_2}) g^{\mu_3\mu_4} \,u \otimes u \otimes \pnu u \\
&& 
+4g^{\mu\mu_1} g^{\nu\mu_2}(\pnu g^{\mu_3\mu_4}) \,u \otimes \pmu u \otimes  u \\ 
&& 
+4g^{\mu\mu_1} g^{\nu\mu_2} g^{\mu_3\mu_4} \,u \otimes \pmu u \otimes  \pnu u
\\[5pt]
4 \bar u^\mu \otimes \bar u^\nu \otimes \pmu \pnu H &\to  &
+ 4 g^{\mu\mu_1} g^{\nu\mu_2} (\pmu\pnu g^{\mu_3\mu_4}) \,u \otimes u \otimes   u \\
&& 
+ 4 g^{\mu\mu_1} g^{\nu\mu_2} (\pmu g^{\mu_3\mu_4}) \,u \otimes u \otimes  \pnu u \\
&& 
+ 4 g^{\mu\mu_1} g^{\nu\mu_2} (\pnu g^{\mu_3\mu_4}) \,u \otimes  u \otimes \pmu u \\
&& 
+ 4 g^{\mu\mu_1} g^{\nu\mu_2} g^{\mu_3\mu_4} \,u \otimes  u \otimes \pmu \pnu u 
\end{array}
\end{align*}

\underline{Variables in $f_4$} without the commun factor $\Pi_{i=1}^6 \xi_{\mu_i}$ and summation over the $\mu_i$:
\begin{align*}\def\arraystretch{1.1}
\begin{array}{lcl}
-4\bar u^\mu \otimes \pmu H \otimes \bar u ^\nu \otimes \pnu H &\to & 
-4g^{\mu\mu_1}(\pmu g^{\mu_2\mu_3}) g^{\nu\mu_4}(\pnu g^{\mu_5\mu_6}) \, u \otimes u \otimes u \otimes u \\
&& 
-4g^{\mu\mu_1}(\pmu g^{\mu_2\mu_3}) g^{\nu\mu_4} g^{\mu_5\mu_6} \, u \otimes u \otimes u \otimes \pnu u \\
&& 
-4g^{\mu\mu_1} g^{\mu_2\mu_3} g^{\nu\mu_4}(\pnu g^{\mu_5\mu_6}) \, u \otimes \pmu u \otimes u \otimes u \\
&& 
-4g^{\mu\mu_1} g^{\mu_2\mu_3} g^{\nu\mu_4}g^{\mu_5\mu_6} \, u \otimes \pmu u \otimes u \otimes \pnu u
\\[5pt]
-4 \bar u ^\mu \otimes \bar u^\nu \otimes \pmu H\otimes \pnu H &\to & 
-4g^{\mu\mu_1} g^{\nu\mu_2} (\pmu g^{\mu_3\mu_4})(\pnu g^{\mu_5\mu_6}) \,  u \otimes u \otimes u \otimes  u \\
&& 
-4g^{\mu\mu_1} g^{\nu\mu_2} (\pmu g^{\mu_3\mu_4})g^{\mu_5\mu_6} \, u \otimes u \otimes u \otimes  \pnu u \\
&& 
-4g^{\mu\mu_1} g^{\nu\mu_2} g^{\mu_3\mu_4} (\pnu g^{\mu_5\mu_6}) \,  u \otimes u \otimes \pmu u \otimes   u \\
&& 
-4g^{\mu\mu_1} g^{\nu\mu_2} g^{\mu_3\mu_4}  g^{\mu_5\mu_6} \, u \otimes u \otimes \pmu u \otimes  \pnu u 
\\[5pt]
-4 \bar u ^\mu \otimes \bar u^\nu \otimes \pnu H\otimes \pmu H &\to & 
-4g^{\mu\mu_1} g^{\nu\mu_2} (\pnu g^{\mu_3\mu_4})(\pmu g^{\mu_5\mu_6}) \,  u \otimes u \otimes u \otimes  u \\
&& 
-4g^{\mu\mu_1} g^{\nu\mu_2} (\pnu g^{\mu_3\mu_4}) g^{\mu_5\mu_6} \,  u \otimes u \otimes u \otimes \pmu u \\
&& 
-4g^{\mu\mu_1} g^{\nu\mu_2}  g^{\mu_3\mu_4}(\pmu g^{\mu_5\mu_6}) \,  u \otimes u \otimes \pnu u \otimes  u \\
&& 
-4g^{\mu\mu_1} g^{\nu\mu_2}  g^{\mu_3\mu_4} g^{\mu_5\mu_6} \,  u \otimes u \otimes \pnu u \otimes  \pmu u .
\end{array}
\end{align*}

A second and tedious step is now to do in \eqref{terme de base} the metric contractions $G(g)_{\mu_1\dots \mu_{2p}} \,\BB_k^{\mu_1\dots \mu_{2p}}$ for previous terms where the $G(g)_{\mu_1\dots \mu_{2p}}$ are given by:
\begin{align}
& G(g)_{\mu_1\mu_2}=\tfrac{1}{2}\, g_{\mu_1\mu_2} \,, \label{G(g)1}\\
& G(g)_{\mu_1\mu_2\mu_3\mu_4} = \tfrac{1}{4} \, (g_{\mu_1\mu_2}g_{\mu_3\mu_4}+ g_{\mu_1\mu_3}g_{\mu_2\mu_4}+g_{\mu_1\mu_4}g_{\mu_2\mu_3})\,,\label{G(g)2}\\
& G(g)_{\mu_1\mu_2\mu_3\mu_4\mu_5\mu_6} = \tfrac{1}{8}\,\,\big[+g_{\mu_1\mu_2}g_{\mu_3\mu_4}g_{\mu_5\mu_6}+g_{\mu_1\mu_2}g_{\mu_3\mu_5}g_{\mu_4\mu_6} +g_{\mu_1\mu_2}g_{\mu_3\mu_6}g_{\mu_4\mu_5} \nonumber\\
 & \hspace{3.95cm}+g_{\mu_1\mu_3}g_{\mu_2\mu_4}g_{\mu_5\mu_6}+g_{\mu_1\mu_3}g_{\mu_2\mu_5}g_{\mu_4\mu_6}+g_{\mu_1\mu_3}g_{\mu_2\mu_6}g_{\mu_4\mu_5}\nonumber\\
 & \hspace{3.95cm} +g_{\mu_1\mu_4}g_{\mu_2\mu_3}g_{\mu_5\mu_6} +g_{\mu_1\mu_4}g_{\mu_2\mu_5}g_{\mu_3\mu_6} + g_{\mu_1\mu_4}g_{\mu_2\mu_6}g_{\mu_3\mu_5} \nonumber\\
 & \hspace{4cm}+ g_{\mu_1\mu_5}g_{\mu_2\mu_3}g_{\mu_4\mu_6} +g_{\mu_1\mu_5}g_{\mu_2\mu_4}g_{\mu_3\mu_6} +g_{\mu_1\mu_5}g_{\mu_2\mu_6}g_{\mu_3\mu_4} \nonumber \\
 & \hspace{3.95cm} +g_{\mu_1\mu_6}g_{\mu_2\mu_3}g_{\mu_4\mu_5} + g_{\mu_1\mu_6}g_{\mu_2\mu_4}g_{\mu_3\mu_5}+ g_{\mu_1\mu_6}g_{\mu_2\mu_5}g_{\mu_3\mu_4} \big]\,. \label{G(g)3}
 \end{align}
 \\
Keeping the same order already obtained in the first step, we get after the contactions:

\underline{Contribution of $f_1$ variable:} $w$ (no contraction since $p=0$).

\underline{Contribution of $f_2$ variables:} 
\begin{multline*}
-\tfrac 12 g^{\mu\nu} g_{\rho\sigma}(\pmu\pnu g^{\rho\sigma})\, u\otimes u 
-g^{\mu\nu} g_{\rho\sigma}(\pnu g^{\rho\sigma}) \,u\otimes \pmu u 
-\tfrac{d}{2} g^{\mu\nu}  \, u \otimes \pmu\pnu u 
\\
-\tfrac 12 g_{\rho\sigma} (\pmu g^{\rho\sigma})\, v^\mu \otimes u
- \tfrac{d}{2} \,v^\mu \otimes \pmu u 
 -\tfrac{1}{2} g_{\mu\nu} \,v^\mu \otimes v^\nu - u \otimes \pmu v^\mu.
\end{multline*}

\underline{Contribution of $f_3$ variables:} 
\begin{align*}
& 
\big[
\tfrac{1}{2} g^{\mu\nu} g_{\rho\sigma} g_{\alpha\beta} (\pmu g^{\rho\sigma})(\pnu g^{\alpha\beta}) 
+ g^{\mu\nu} g_{\rho\sigma} g_{\alpha\beta} (\pmu g^{\rho\alpha})(\pnu g^{\sigma\beta})
\big]
\,u\otimes u\otimes u  
\\
&
+\tfrac{d+2}{2} g^{\mu\nu} g_{\rho\sigma} (\pnu g^{\rho\sigma})\,u\otimes u\otimes \pmu u 
+\tfrac{d+2}{2} g^{\mu\nu} g_{\rho\sigma} (\pnu g^{\rho\sigma})\,u\otimes \pmu u\otimes  u  
+\tfrac{d(d+2)}{2} g^{\mu\nu} \,u\otimes \pmu u\otimes \pnu u 
\\
&
+ \tfrac{1}{2}  g_{\rho\sigma} (\pmu g^{\rho\sigma})\,v^\mu\otimes  u\otimes  u  
+ g_{\mu\nu}(\prho g^{\rho\nu}) \, v^\mu \otimes u \otimes u 
+ \tfrac{d+2}{2} \,v^\mu \otimes u \otimes \pmu u
\\
&
+ \tfrac{1}{2} g_{\rho\sigma} (\pmu g^{\rho\sigma})\,u\otimes u \otimes v^\mu 
+ g_{\mu\nu}(\prho g^{\rho\nu})\,u\otimes u \otimes v^\mu 
+ \tfrac{d+2}{2} \,u\otimes \pmu u \otimes v^\mu
\\
& 
+\tfrac{1}{2} g_{\rho\sigma} (\pmu g^{\rho\sigma})\,u\otimes  v^\mu\otimes  u  
+ g_{\mu\nu}(\prho g^{\rho\nu}) \, u \otimes v^\mu \otimes u 
+ \tfrac{d+2}{2} \,u \otimes v^\mu \otimes \pmu u 
\\
& 
+\big[
g_{\rho\sigma} (\pmu g^{\mu\nu})(\pnu g^{\rho\sigma})
+2 g_{\rho\sigma}(\pmu g^{\nu\rho})(\pnu g^{\mu\sigma})
\big] u\otimes u\otimes u 
+ (d+2) (\pnu g^{\mu\nu}) \,u\otimes u \otimes \pmu u 
\\
& 
+[g^{\mu\nu}g_{\rho\sigma} (\pnu g^{\rho\sigma}) +2 (\pnu g^{\mu\nu})]\, u\otimes \pmu u \otimes u 
+ (d+2) g^{\mu\nu} \, u \otimes \pmu u \otimes \pnu u 
\\
& 
+ [g^{\mu\nu}g_{\rho\sigma} (\pmu\pnu g^{\rho\sigma}) +2 (\pmu\pnu g^{\mu\nu})]\, u\otimes u \otimes u  
+ [g^{\mu\nu}g_{\rho\sigma} (\pnu g^{\rho\sigma}) +2 (\pnu g^{\mu\nu})]\, u\otimes u \otimes \pmu u 
\\
&
+ [g^{\mu\nu}g_{\rho\sigma} (\pnu g^{\rho\sigma}) +2 (\pnu g^{\mu\nu})]\, u\otimes u \otimes  \pmu u  
+ (d+2) g^{\mu\nu}\, u\otimes u \otimes \pmu\pnu u.
\end{align*}
which, once collected, gives
\begin{align*}
&
\big[
g^{\mu\nu}g_{\rho\sigma} (\pmu\pnu g^{\rho\sigma}) 
+2 (\pmu\pnu g^{\mu\nu})
+g_{\rho\sigma} (\pmu g^{\mu\nu})(\pnu g^{\rho\sigma})
+2 g_{\rho\sigma}(\pmu g^{\nu\rho})(\pnu g^{\mu\sigma})
\\
& \quad\quad\quad\quad
\tfrac{1}{2} g^{\mu\nu} g_{\rho\sigma} g_{\alpha\beta} (\pmu g^{\rho\sigma})(\pnu g^{\alpha\beta}) 
+ g^{\mu\nu} g_{\rho\sigma} g_{\alpha\beta} (\pmu g^{\rho\alpha})(\pnu g^{\sigma\beta})
\big]
\,u\otimes u\otimes u 
\\
&
+(d+6) \big[
\tfrac 12 g^{\mu\nu}g_{\rho\sigma} (\pnu g^{\rho\sigma}) +  (\pnu g^{\mu\nu})
\big]
\,u\otimes u\otimes \pmu u 
\\
&
+ \big[
\tfrac{d+4}{2} g^{\mu\nu}g_{\rho\sigma} (\pnu g^{\rho\sigma}) + 2 (\pnu g^{\mu\nu})
\big]
\,u\otimes \pmu u\otimes u 
\\
&
+ \tfrac{(d+2)^2}{2} g^{\mu\nu} \,u\otimes \pmu u\otimes \pnu u 
+ (d+2) g^{\mu\nu}\, u\otimes u \otimes \pmu\pnu u
\\
&
+ \big[
\tfrac{1}{2} g_{\rho\sigma} (\pmu g^{\rho\sigma}) + g_{\mu\nu}(\prho g^{\rho\nu})
\big]
(v^\mu \otimes u \otimes u + u \otimes v^\mu \otimes u + u \otimes u \otimes v^\mu)
\\
&
+ \tfrac{d+2}{2} (v^\mu \otimes u \otimes \pmu u + u \otimes \pmu u \otimes v^\mu  + u \otimes v^\mu \otimes \pmu u)
\end{align*}

\underline{Contribution of $f_4$ variables:} 
We use the following symmetry: in previous three terms of $f_4$, one goes from the first to the second right terms by the change $(\mu_2,\mu_3,\mu_4) \to (\mu_3,\mu_4,\mu_2)$ and from the second to the third terms via $(\mu,\nu) \to (\nu,\mu)$ and $(\mu_1,\mu_2) \to (\mu_2,\mu_1)$. So after the contraction of the first term and using that symmetry (which explains the factors 3 and 2), we get
\begin{align*}
&\big[ 
-\tfrac 12 g^{\mu\nu} g_{\rho\sigma} g_{\alpha\beta} (\pmu g^{\rho\sigma}) (\pnu g^{\alpha\beta}) 
- g^{\mu\nu} g_{\rho\sigma} g_{\alpha\beta} (\pmu g^{\rho\alpha}) (\pnu g^{\sigma\beta})
-2 g_{\rho\sigma} (\pmu g^{\mu\nu}) (\pnu g^{\rho\sigma})
\\
& \hspace{1cm}
-2 g_{\rho\sigma} (\pmu g^{\mu\rho}) (\pnu g^{\nu\sigma})
-2 g_{\rho\sigma} (\pmu g^{\nu\rho}) (\pnu g^{\mu\sigma})
\big] \, u \otimes u \otimes u \otimes u 
\\
& 
- (d+4) \big[(\pmu g^{\mu\nu}) +\tfrac 12 g^{\mu\nu} g_{\rho\sigma}(\pmu g^{\rho\sigma}) \big] 
\big(
3 \, u\otimes u \otimes u\otimes \pnu u\\
& \hspace{6.6cm}
+ 2 \,u\otimes u \otimes \pnu u\otimes u
+ u \otimes \pnu u \otimes u \otimes u
\big)
\\
& 
-\tfrac 12 (d+4)(d+2) g^{\mu\nu} \, 
\big(2\,u \otimes  u \otimes \pnu u \otimes \pnu u + u \otimes \pmu u \otimes u \otimes \pnu u \big).
\end{align*}

It worth to mention that all results of this section are valid in arbitrary dimension $d$ of the manifold; this explains why we have not yet replaced $d$ by 4.

%%%%%%%%%%%%%%%%%%%%%%%%%%%%%%
\subsection{\texorpdfstring{Application of operators $I_{d/2+p,k}$ for $d=4$}{Application of operators I for d=4}}
%%%%%%%%%%%%%%%%%%%%%%%%%%%%%%

We can now compute in \eqref{terme de base} the application of $I_{d/2+p,k}$ on each previous $ G(g)_{\mu_1\dots \mu_{2p}} \,\BB_k^{\mu_1\dots \mu_{2p}}$. Lemma \ref{number to calculate} and \eqref{Tkp example} tell us that we need 
\begin{align*}
& I_{2,1}=[R_0(u)\,R_1(u)]^{-1}, \\
& I_{3,2}  = \tfrac 12 \,[R_0(u)\,R_1(u)\,R_2(u)]^{-1}, \\
& I_{4,3} = \tfrac 16 \,[R_0(u)\,R_1(u)\,R_2(u)\,R_3(u))]^{-1},\\
& I_{5,4} =  \tfrac{1}{24} \,[R_0(u)\,R_1(u)\,R_2(u)\,R_3(u)\,R_4(u)]^{-1}
\end{align*}
written here for $d=4$.

We give below the list of all $(\tr \circ\, m \circ \kappa^* \circ I_{d/2+p,k}) [\BB_k]$ for each kind of $\BB_k \in \hH_k$ found previously, respecting again the order. The cyclicity of the trace is used.

\underline{Contribution of $f_1$ variable:} 
\begin{align*}
w \to \tr[u^{-2}w].
\end{align*}

\underline{Contribution of $f_2$ variables:} 
\begin{align*}%\def\arraystretch{1.3}
%\begin{array}{ll}
 u \otimes u &\to \tfrac 12 \tr[u^{-1}] \\
 u \otimes \pnu u &\to \tfrac 12  \tr[u^{-2}(\pnu u)] \\
 u \otimes \pmu\pnu u &\to \tfrac 12  \tr[u^{-2}(\pmu\pnu u)] \\
 v^\mu \otimes u &\to \tfrac 12 \tr[u^{-2}v^\mu] \\
 v^\mu \otimes \pmu u &\to \tfrac 12 \tr[u^{-2}v^\mu u^{-1} (\pmu u)] \\
v^\mu \otimes v^\nu &\to \tfrac 12 \tr[u^{-2}v^\mu u^{-1} v^\nu] \\
u \otimes \pmu v^\mu &\to \tfrac 12 \tr[u^{-2}(\pmu v^\mu)].
%\end{array}
\end{align*}

\underline{Contribution of $f_3$ variables:} 
\begin{align*}%\def\arraystretch{1.3}
%\begin{array}{ll}
u \otimes u \otimes u &\to \tfrac 16 \tr[u^{-1}] \\
u \otimes u \otimes \pnu u &\to \tfrac 16 \tr[u^{-2}(\pnu u)] \\
u \otimes \pnu u \otimes  u &\to \tfrac 16 \tr[u^{-2}(\pnu u)] \\
u \otimes \pmu u \otimes  \pnu u &\to \tfrac 16 \tr[u^{-2}(\pmu u) u^{-1}(\pnu u)] \\
v^\mu \otimes u \otimes u &\to \tfrac 16 \tr[u^{-2}v^\mu] \\
v^\mu \otimes u \otimes  \pmu u& \to \tfrac 16 \tr[u^{-2}v^\mu u^{-1}(\pmu u)] \\
u \otimes u \otimes v^\mu &\to \tfrac 16 \tr[u^{-2}v^\mu] \\
u \otimes \pmu u \otimes v^\mu &\to \tfrac 16 \tr[u^{-2}(\pmu u) u^{-1}v^\mu] \\
u \otimes v^\mu \otimes u &\to \tfrac 16 \tr[u^{-2}v^\mu] \\
u \otimes v^\mu \otimes \pmu u &\to \tfrac 16 \tr[u^{-2}v^\mu u^{-1} (\pmu u)] \\
u \otimes u \otimes \pmu\pnu u &\to \tfrac 16 \tr[u^{-2}(\pmu\pnu u)].
%\end{array}
\end{align*}

\underline{Contribution of $f_4$ variables:} 
\begin{align*}%\def\arraystretch{1.3}
%\begin{array}{ll}
 u \otimes u \otimes u \otimes u  & \to \tfrac{1}{24} \tr[u^{-1}] \\
 u \otimes u \otimes u \otimes \pnu u  & \to \tfrac{1}{24} \tr[u^{-2}(\pnu u)] \\
 u \otimes \pnu u \otimes u \otimes  u  & \to \tfrac{1}{24} \tr[u^{-2}(\pnu u)] \\
 u \otimes u \otimes \pnu u \otimes  u  & \to \tfrac{1}{24} \tr[u^{-2}(\pnu u)] \\
 u \otimes \pmu u \otimes u \otimes \pnu u  & \to \tfrac{1}{24} \tr[u^{-2}(\pmu u)u^{-1} (\pnu u)] \\
 u \otimes u \otimes \pmu u \otimes \pnu u & \to \tfrac{1}{24} \tr[u^{-2}(\pmu u)u^{-1} (\pnu u)].
%\end{array}
\end{align*}

%%%%%%%%%%%%%%%%%%%%%%%%%%%%%%%%%%%%
\subsection{Main results}
%%%%%%%%%%%%%%%%%%%%%%%%%%%%%%%%%%%%

The recollection of all contributions \eqref{terme de base} for $a_1(x)$ is now ready for the first interesting result:

\begin{theorem}
\label{calcul de a1}
Assume that $P = -(u \,g^{\mu\nu}\pmu\pnu +v^\nu\pnu +w)$ is a selfadjoint elliptic operator acting on $L^2(M,V)$ for a 4-dimensional boundaryless Riemannian manifold $(M,g)$ and a vector bundle $V$ over $M$ where $u, v^\mu, w$ are local maps on $M$ with values in $M_N$, with $u$ positive and invertible. Then, its local $a_1(x)$ heat-coefficient in \eqref{heat-trace-asympt} for $x\in M$ is
\begin{equation}
\label{valeur de a1}
a_1 =  g_4 (\alpha \tr [u^{-1}] +  \tr[u^{-2} \ell])
\end{equation}
where $g_4=\tfrac{\abs{g}^{1/2}}{16 \pi^2}$ and
\begin{align*}
\alpha &\vc 
\tfrac 13 (\pmu\pnu  g^{\mu\nu}) 
-\tfrac 12 g^{\mu\nu} g_{\rho \sigma}(\pmu\pnu g^{\rho\sigma}) 
+ \tfrac{1}{48} g^{\mu\nu} g_{\rho \sigma} g_{\alpha\beta}(\pmu g^{\rho\sigma})(\pnu g^{\alpha\beta}) 
\\
& \qquad\qquad
+ \tfrac{1}{24} g^{\mu\nu} g_{\rho \sigma} g_{\alpha\beta}(\pmu g^{\rho\alpha})(\pnu g^{\sigma\beta}) 
-\tfrac{1}{12} g_{\rho\sigma}(\pmu g^{\mu\nu})(\pnu g^{\rho\sigma}) 
\\
& \qquad\qquad 
+ \tfrac{1}{12} g_{\rho\sigma}(\pmu g^{\nu\rho})(\pnu g^{\mu\sigma}) 
- \tfrac 14 g_{\rho\sigma}(\pmu g^{\mu\rho})(\pnu g^{\nu\sigma}), 
\\[5pt]
\ell &\vc 
w 
+ \tfrac 12 g_{\mu\nu}(\prho g^{\rho\nu}) \,v^\mu  
-\tfrac 12  (\pmu v^\mu)
- \tfrac 14 g_{\mu\nu} \,v^\mu u^{-1}v^\nu 
+ \tfrac 12 \,(\pmu u)u^{-1} v^\mu\,.
\end{align*}
\end{theorem}

Since the operator $P$ is not written in terms of objects which have simple (homogeneous) transformations by change of coordinates and gauge transformation, this result does not make apparent any explicit Riemannian or gauge invariant expressions. For example $\ell$ is not a scalar under a change of coordinates since $v^\mu$ is not a vector. This is why we have not used normal coordinates until now. Nevertheless, from Lemma~\ref{lem-u vmu w cc gt}, one can deduce:

\begin{lemma}
Under a gauge transformation, $\ell$ in \eqref{valeur de a1} transforms as $\ell \to \gamma \ell \gamma^{-1}$. This checks that $a_1(x)$ is gauge invariant.
\end{lemma}

As shown in \ref{subsection-diff-gauge-P}, with the help of a gauge connection $A_\mu$, one can change the variables $(u,v^\mu, w)$ to variables $(u, p^\mu, q)$ well adapted to change of coordinates and gauge transformations (see \eqref{eq-ccupq} and \eqref{eq-gtupq}). For $u^{\mu\nu} = g^{\mu \nu} u$, \eqref{vmu(A,p,q)} and \eqref{w(A,p,q)} becomes
\begin{align}
v^\mu &= 
\big[-\tfrac 12 g^{\mu\nu} g_{\rho\sigma} (\pnu g^{\rho\sigma}) + \pnu g^{\mu\nu} \big] u
+ g^{\mu\nu} (\pnu u)
+ g^{\mu\nu} (u A_\nu + A_\nu u)
+ p^\mu
\label{vmu(A,p,q) for u}
\\
w &= 
\big[-\tfrac 12 g^{\mu\nu} g_{\rho\sigma} (\pnu g^{\rho\sigma}) + \pnu g^{\mu\nu} \big] u A_\mu
+ g^{\mu\nu} (\pnu u) A_\mu
+ g^{\mu\nu} u (\pmu A_\nu)
\nonumber\\
&\quad\quad\quad\quad
+ g^{\mu\nu} A_\mu u A_\nu
+ p^\mu A_\mu
+ q
\label{w(A,p,q) for u}.
\end{align}
Relations \eqref{vmu(A,p,q) for u} and \eqref{w(A,p,q) for u} can be injected into \eqref{valeur de a1} to get an explicitly diffeomorphism and gauge invariant expression. In order to present the result of this straightforward computation, let us introduce the following notations.

Given the Christoffel symbols $\Gamma_{\mu\nu}^\rho \vc \tfrac 12 g^{\rho \sigma}(\partial_\mu g_{\sigma\nu} +\partial_\nu g_{\sigma\mu} -\partial_\sigma g_{\mu\nu})$, 
the Riemann curvature tensor $R_{\beta\mu\nu}^\alpha \vc \partial_\mu \Gamma_{\beta \nu}^\alpha -\partial _\nu \Gamma_{\beta\mu}^\alpha + \Gamma_{\mu \rho}^\alpha \Gamma_{\beta \nu}^\rho - \Gamma_{\nu\rho}^\alpha\Gamma_{\beta\mu}^\rho$, 
and the Ricci tensor $R_{\mu\nu} \vc R_{\mu\rho\nu}^\rho$, 
the scalar curvature $R\vc g^{\mu\nu}R_{\mu\nu}$ 
computed in terms of the derivatives of the inverse metric is 
\begin{multline}
\label{Scalar curvature}
R = 
g^{\mu\nu} g_{\rho\sigma} (\pmu\pnu g^{\rho\sigma})
- (\pmu\pnu g^{\mu\nu}) 
+ g_{\rho\sigma}(\pmu g^{\mu\nu})(\pnu g^{\rho\sigma}) 
+\tfrac12 g_{\rho\sigma}(\pmu g^{\nu\rho})(\pnu g^{\mu\sigma}) 
\\
-\tfrac 14 g^{\mu\nu} g_{\rho\sigma} g_{\alpha\beta} (\pmu g^{\rho\sigma})(\pnu g^{\alpha\beta}) 
-\tfrac 54 g^{\mu\nu} g_{\rho\sigma} g_{\alpha\beta} (\pmu g^{\rho\alpha})(\pnu g^{\sigma\beta}),
\end{multline}
and one has
\begin{align*}
g^{\mu\nu} \Gamma_{\mu\nu}^\rho = \tfrac 12 g^{\rho\sigma} g_{\alpha\beta} (\psigma g^{\alpha\beta}) - \psigma g^{\rho\sigma},
\qquad \Gamma_{\sigma\rho}^\sigma &= -\tfrac 12 g_{\alpha\beta} (\prho g^{\alpha\beta}).
\end{align*}
Let $\nabla_\mu$ be the (gauge) covariant derivative on $V$ (and its related bundles): 
$$
\nabla_\mu s \vc \partial_\mu s + A_\mu s\quad \text{for any section }s \text{ of }V.
$$
From \eqref{eq-gtupq}, $u$, $p^\mu$ and $q$ are sections of the endomorphism vector bundle $\End(V) = V^\ast \otimes V$ (while $v^\mu$ and $w$ are not), so that $\nabla_\mu u = \pmu u + [A_\mu, u]$ (and the same for $p^\mu$ and $q$). We now define $\hnabla_\mu$, which combines $\nabla_\mu$ and the linear connection induced by the metric $g$:
\begin{gather*}
\hnabla_\mu u \vc \pmu u + [A_\mu, u] = \nabla_\mu u,
\quad\quad\quad
\hnabla_\mu p^\rho \vc \pmu p^\rho + [A_\mu, p^\rho] + \Gamma^\rho_{\mu\nu} p^\nu = \nabla_\mu p^\rho + \Gamma^\rho_{\mu\nu} p^\nu
\\
\hnabla_\mu \hnabla_\nu u \vc \pmu \hnabla_\nu u + [A_\mu, \hnabla_\nu u] - \Gamma^\rho_{\mu\nu} \hnabla_\rho u
= \pmu \nabla_\nu u + [A_\mu, \nabla_\nu u] - \Gamma^\rho_{\mu\nu} \nabla_\rho u
\end{gather*}
so that
\begin{align*}
g^{\mu\nu} \hnabla_\mu \hnabla_\nu u &= g^{\mu\nu} \nabla_\mu \nabla_\nu u - \big[ \tfrac 12 g^{\mu\nu} g_{\alpha\beta} (\pmu g^{\alpha\beta}) - \pmu g^{\mu\nu} \big] \nabla_\nu u,
\\
\hnabla_\mu p^\mu &= \nabla_\mu p^\mu - \tfrac 12 g_{\alpha\beta} (\pmu g^{\alpha\beta}) p^\mu.
\end{align*}
Any relation involving $u, p^\mu, q, g$ and $\hnabla_\mu$ inherits the homogeneous transformations by change of coordinates and gauge transformations of these objects.

Let us state now the result of the computation of $a_1(x)$ in terms of $(u, p^\mu, q)$:
\begin{theorem}
\label{calcul de a1 en u pmu q}
Assume that $P = -(\abs{g}^{-1/2} \nmu \abs{g}^{1/2} g^{\mu\nu} u \nnu + p^\mu \nmu +q)$ is a selfadjoint elliptic operator acting on $L^2(M,V)$ for a 4-dimensional boundaryless Riemannian manifold $(M,g)$ and a vector bundle $V$ over $M$ where $u,\,p^\mu,\,q$ are sections of endomorphisms on $V$ with $u$ positive and invertible. Then, its local $a_1(x)$ heat-coefficient in \eqref{heat-trace-asympt} for $x\in M$ is
\begin{multline}
\label{valeur de a1 en u pmu q}
a_1 =  g_4 \big( 
\tfrac 16 R \tr[u^{-1}] + \tr[u^{-2} q]
\\
- \tfrac 12 g^{\mu\nu}  \tr[u^{-2} \hnabla_\mu \hnabla_\nu u ]
- \tfrac 12 \tr[u^{-2}  \hnabla_\mu p^\mu ]
\\
+ \tfrac 14 g^{\mu\nu}  \tr[u^{-2}  (\hnabla_\mu u - g_{\mu\rho} p^\rho) u^{-1} (\hnabla_\nu u + g_{\nu\sigma} p^\sigma) ]
\big)
\end{multline}
where $g_4=\tfrac{\abs{g}^{1/2}}{16 \pi^2}$.
\end{theorem}

\begin{proof}
This can be checked by an expansion of the RHS of \eqref{valeur de a1 en u pmu q}. A more subtle method goes using normal coordinates in \eqref{valeur de a1}, \eqref{vmu(A,p,q) for u}, \eqref{w(A,p,q) for u}, knowing that $a_1(x)$ is a scalar and $(u,\,p^\mu,\, q)$ are well adapted to change of coordinates.
\end{proof}

\begin{remark}
For the computation of $a_r(x)$ with $r\geq 2$, directly in terms of the variables $(u,\,p^\mu,\, q)$, the strategy is to use normal coordinates from the very beginning, which simplifies the computations of terms $\BB_k^{\mu_1\dots \mu_{2p}}$ of \eqref{terme de base}. Then an equivalent result to Theorem \ref{calcul de a1} would be obtained, but only valid in normal coordinates. Then by the change of variables \eqref{vmu(A,p,q) for u}, \eqref{w(A,p,q) for u}, a final result as Theorem \ref{calcul de a1 en u pmu q} would be calculated.
\end{remark}

%%%%%%%%%%%%%%%%%%%%%%%%%%%%%%%%%%%%%%%%%
\subsection{A dimension independent result}
%%%%%%%%%%%%%%%%%%%%%%%%%%%%%%%%%%%%%%%%%

We found in Theorem~\ref{calcul de a1 en u pmu q} that the coefficient of $\Tr[u^{-1}]$ is $g_4\, \tfrac 16 R$ in dimension 4. In fact, it is a general result, which is not a surprise:

\begin{proposition}
\label{prop 1 sur 6 R}
In any dimension $d \geq 1$, the coefficient in $a_1$ of $\Tr[u^{1-d/2}]$, written in variables $(u, p^\mu, q)$, is $g_d\,  \tfrac 16 R$.
\end{proposition}

\begin{proof}
Since $r=1$, we already know that $f_k$ is associated to $I_{d/2+k,k+1}$ and $1\leq k \leq 4$.\\
By the very definition, the operators $I_{\alpha,k}$ applied on $\bbbone \otimes u\otimes \cdots \otimes u$ ($k$ factors $u$) are easily computed since all $R_i(u)$ commute with $u$, so that 
$$
(m\circ I_{d/2+k-1,k})[\bbbone \otimes \underbrace{u\otimes \cdots \otimes u}_{k}]= \text{Vol}(\Delta_k)\, u^{-(d/2+k-1)+k}=\tfrac{1}{k!}\,u^{-d/2+1}.
$$
From the results obtained in \ref{subsection-collectcontraction}, we swap $(u, v^\mu, w)$ with the new variables $(u, p^\mu, q)$ as in \eqref{vmu(A,p,q) for u} and \eqref{w(A,p,q) for u}, and then collect only the terms $u\otimes \cdots \otimes u$ ($k$ factors $u$).\\
\underline{Contribution of $f_2$ giving $u\otimes u$:} (with their origin before the change of variables)
\begin{align*}
u\otimes u & \to
-\tfrac{1}{2} g^{\mu\nu} g_{\rho\sigma}(\pmu\pnu  g^{\rho\sigma})
\\
v^\mu\otimes u & \to 
-\tfrac{1}{2}  g_{\rho\sigma}(\pmu  g^{\rho\sigma})[-\tfrac{1}{2} g^{\mu\nu} g_{\alpha\beta}(\pnu  g^{\alpha\beta})+(\pnu g^{\mu\nu})] 
\\
v^\mu \otimes v^\nu & \to 
-\tfrac 12 g_{\mu\nu}
[-\tfrac 12 g^{\mu\sigma}g_{\alpha\beta}(\psigma g^{\alpha\beta})+(\psigma g^{\sigma \mu})]
[-\tfrac 12 g^{\nu\rho}g_{\gamma\delta}(\prho g^{\gamma\delta})+(\prho g^{\rho\nu})]
\\
u\otimes \pmu v^\mu & \to 
+ \tfrac 12 g_{\alpha\beta} (\pmu g^{\mu\nu})(\pnu g^{\alpha\beta})
+ \tfrac 12 g^{\mu\nu}(\pmu g_{\alpha\beta})(\pnu g^{\alpha\beta})
+ \tfrac 12 g^{\mu\nu}g_{\alpha\beta}(\pmu\pnu g^{\mu\sigma})
 -\pmu\pnu g^{\mu\nu}\,.
\end{align*}
\underline{Contribution of $f_3$ giving $u\otimes u\otimes u$:} 
\begin{align*}
& +\tfrac12 g^{\mu\nu} g_{\rho \sigma} g_{\alpha\beta}(\pmu g^{\rho\sigma})(\pnu g^{\alpha\beta}) + g^{\mu\nu} g_{\rho \sigma} g_{\alpha\beta}(\pmu g^{\rho\alpha})(\pnu g^{\sigma\beta}) \\
& 
+ 3 \,
[\tfrac 12 g_{\rho\sigma}(\pmu g^{\rho\sigma}) + g_{\mu\sigma} (\prho g^{\rho\sigma})]
\,
[ -\tfrac 12 g^{\mu\nu} g_{\alpha\beta} (\pnu g^{\alpha\beta})+ (\pnu g^{\mu\nu})]
\\
& 
+ g_{\rho\sigma}(\pmu g^{\mu\nu})(\pnu g^{\rho\sigma}) +2 g_{\rho\sigma}(\pmu g^{\nu\rho})(\pnu g^{\mu \sigma}) 
\\
& 
+g^{\mu\nu} g_{\rho\sigma}(\pmu\pnu  g^{\rho\sigma}) +2 (\pmu\pnu g^{\mu\nu}) 
\end{align*}
\underline{Contribution of $f_4$ giving $u\otimes u\otimes u\otimes u$:} 
\begin{multline*}
3\,\big[ 
-\tfrac12 g^{\mu\nu} g_{\rho \sigma} g_{\alpha\beta}(\pmu g^{\rho\sigma})(\pnu g^{\alpha\beta}) 
-g^{\mu\nu} g_{\rho \sigma} g_{\alpha\beta}(\pmu g^{\rho\alpha})(\pnu g^{\sigma\beta})
-2 g_{\rho\sigma}(\pmu g^{\mu\nu})(\pnu g^{\rho\sigma})
\\
 -2g_{\rho\sigma}(\pmu g^{\mu\rho})(\pnu g^{\nu\sigma}) 
 -2 g_{\rho\sigma}(\pmu g^{\nu\rho})(\pnu g^{\mu\sigma})\big].
\end{multline*}
Once these contributions of $f_k$ are multiplied by $\tfrac{1}{k!}$, one checks that their sum is exactly $\tfrac 16 R$ using the scalar curvature $R$ written as in \eqref{Scalar curvature}.
\end{proof}

\begin{remark}
In the present method, the factor $\tfrac 16 R$ is explicitly and straightforwardly computed from the metric entering  $u^{\mu\nu} = g^{\mu\nu} u$, as in \cite{Avramidi2006} for instance. Many methods introduce $R$ using diffeomorphism invariance and compute the coefficient $\tfrac 16$ using some “conformal perturbation” of $P$ (see \cite[Section 3.3]{Gilkeybook1}).
\end{remark}

%%%%%%%%%%%%%%%%%%%%%%%%%%%%%%%%%%%%
\subsection{\texorpdfstring{Case of scalar symbol: $u(x)=f(x)\, \bbbone_N$}{Case u(x)=f(x)1}}
%%%%%%%%%%%%%%%%%%%%%%%%%%%%%%%%%%%%

Let us consider now the specific case $u(x)=f(x)\, \bbbone_N$, where $f$ is a nowhere vanishing positive function. Then \eqref{vmu(A,p,q) for u} simplifies to
\begin{equation*}
v^\mu = 
[-\tfrac 12 g^{\mu\nu} g_{\rho\sigma} (\pnu g^{\rho\sigma}) + \pnu g^{\mu\nu} ] f \, \bbbone_N
+ g^{\mu\nu} (\pnu f)\, \bbbone_N
+ 2 f\,g^{\mu\nu}  A_\nu
+ p^\mu
\end{equation*}
and we can always find $A_\mu$ such that $p^\mu = 0$:
\begin{equation*}
A_\mu = \tfrac 12 \big( g_{\mu\nu} f^{-1} v^\nu + [\tfrac 12 g_{\rho\sigma} (\pmu g^{\rho\sigma}) - g_{\mu\nu} (\prho g^{\rho\nu}) - f^{-1} (\pmu f)]\,  \bbbone_N \big).
\end{equation*}
One can check, using \eqref{eq-vmu cc gt}, that $A_\mu$ satisfies the correct gauge transformations. This means that $P$ can be written as
\begin{equation*}
P = -(\abs{g}^{-1/2} \nmu \abs{g}^{1/2} g^{\mu\nu} f \nnu +q)
\end{equation*}
where the only matrix-dependencies are in $q$ and $A_\mu$.

Since $u$ is in the center, $\nabla_\mu u =( \pmu f) \, \bbbone_N$ and \eqref{valeur de a1 en u pmu q} simplifies as
\begin{equation*}
a_1 =  g_4 \big( 
\tfrac{N}{6} \, f^{-1}\,R + f^{-2} \tr[q]
- \tfrac{N}{2} g^{\mu\nu}  f^{-2} (\pmu \pnu f)
+ \tfrac{N}{2} g^{\mu\nu}  f^{-2} \Gamma^\rho_{\mu\nu} (\prho f)
+ \tfrac{N}{4} g^{\mu\nu}   f^{-3} (\pmu f) (\pnu f)
\big)
\end{equation*}
in which $A_\mu$ does not appears. Now, if $f$ is constant, we get well-known result (see \cite[Section 3.3]{Gilkeybook1}):
\begin{equation*}
a_1 =  g_4 \big(  \tfrac{N}{6}\,f^{-1}\, R + f^{-2} \tr[q] \big).
\end{equation*}

%%%%%%%%%%%%%%%%%%%%%%%%%%%%%%%%%%%%
\section{About the method}
\label{method}
%%%%%%%%%%%%%%%%%%%%%%%%%%%%%%%%%%%%

%%%%%%%%%%%%%%%%%%%%%
\subsection{Existence}
%%%%%%%%%%%%%%%%%%%%%

For the operator $P$ given in \eqref{def-P}, the method used here assumes only the existence of asymptotics \eqref{heat-trace-asympt}. This is the case when $P$ is elliptic and selfadjoint. \\
Selfadjointness of $P$ on $L^2(M, V)$ is not really a restriction since we remark that given an arbitrary  family of $u^{\mu\nu}$ satisfying \eqref{Hyp-defpositiv}, skewadjoint matrices $\tilde{v}^\mu$ and a selfadjoint matrix $\tilde{w}$, we get a formal selfadjoint elliptic operator $P$ defined by \eqref{def-P} where 
\begin{align*}
& v^\mu = \tilde{v}^\mu + (\pnu \log \abs{g}^{1/2}) \,u^{\mu\nu} + \pnu u^{\mu\nu},\\
& w = \tilde{w} + \tfrac 12 [-\pmu \tilde{v}^\mu +(\pmu \log \abs{g}^{1/2}) \,\tilde{v}^\mu].
\end{align*}

A crucial step in our method is to be able to compute the integral \eqref{Tkp versus Ink} for a general $u^{\mu\nu}$. The case $u^{\mu\nu}=g^{\mu\nu}u$ considered in Section \ref{Section-example} makes that integral manageable.

%%%%%%%%%%%%%%%%%%%%
\subsection{\texorpdfstring{On explicit formulae for $u^{\mu\nu}=g^{\mu\nu}u$}{On explicit formulae for umunu}}
%%%%%%%%%%%%%%%%%%%%

For $u^{\mu\nu} = g^{\mu\nu} u$, the proposed method is a direct computational machinery. Since the method can be computerized, this could help to get $a_r$ in Case 2 ($d$ even and $r<d/2$). Recall the steps: 1) expand the arguments of the initial $f_k$'s, 2) contract with the tensor $G(g)$, 3) apply to the corresponding operators $I_{\alpha, k}$, 4) collect all similar terms. Further eventual steps: 5) change variables to $(u, p^\mu, q)$, 6) identify (usual) Riemannian tensors and covariant derivatives (in terms of $A_\mu$ and Christoffel symbols).

Is it possible to get explicit formulae for $a_r$ from the original ingredients $(u,v^\mu,w)$ of $P$? An explicit formula should look like \eqref{valeur de a1} or \eqref{valeur de a1 en u pmu q}. This excludes the use of spectral decomposition of $u$ as in \eqref{integration spectrale} which could not be recombined as
\begin{align*}
\sum_{\text{finite sum}} \tr[h_{(0)}(u) B_1h_{(1)}(u) B_2 \cdots B_k h_{(k)}(u)]
\end{align*}
where the $h_i$ are continuous functions and the $B_i$ are equal to $u,\,v^\mu,\,w$ or their derivatives. The obstruction to get such formula could only come from the operators $I_{\alpha,k}$ and not from the arguments $B_i$. Thus the matter is to understand the $u$-dependence of $I_{\alpha,k}$.

Let us consider $I_{\alpha, k}$ as a map $u \mapsto I_{\alpha, k}(u)$. An operator map $u \mapsto A(u) \in \B(\hH_k)$ is called {\it $u$-factorizable} (w.r.t. the tensor product) if it can be written as
\begin{equation*}
A(u) = \sum_{\text{finite sum}} R_0(h_{(0)}(u)) \, R_1(h_{(1)}(u)) \, \cdots  \, R_k(h_{(k)}(u))
\end{equation*}
where the $h_{(i)}$ are continuous functions on $\R^*_+$.

\begin{lemma}
\label{A(u) factorizable}
Let $u \mapsto A(u) = F(R_0(u), \dots, R_k(u))$ the operator map defined by a continuous function $F : (\R^*_+)^{k+1} \to \R^*_+$. Then $A$ is $u$-factorizable iff $F$ is decomposed as
\begin{equation}
\label{F decomposed}
F(r_0, \dots, r_k) = \sum_{\text{finite sum}} h_{(0)}(r_0) h_{(1)}(r_1) \cdots h_{(k)}(r_k) 
\end{equation}
for continuous functions $h_{(i)}$.
\end{lemma}

\begin{proof}
Let $\lambda_i$ be the eigenvalues of $u$ and $\pi_i$ the associated eigenprojections. If $F$ is decomposed, then by functional calculus, one has
\begin{align*}
A(u) &= \sum_{i_0, \dots, i_k} F(\lambda_{i_0}, \dots, \lambda_{i_k}) \, R_0(\pi_{i_0}) \cdots R_k(\pi_{i_k})
\\
&= 
\sum_{i_0, \dots, i_k} \sum_{\text{finite sum}}  h_{(0)}(\lambda_{i_0}) \cdots h_{(k)}(\lambda_{i_k}) \,  R_0(\pi_{i_0}) \cdots R_k(\pi_{i_k})
\\
&=
\sum_{\text{finite sum}} R_0(h_{(0)}(u)) \, \cdots  \, R_k(h_{(k)}(u)).
\end{align*}
If $A$ is $u$-factorizable, then this computation can be seen in the other way around to show that $F$ is decomposed.
\end{proof}

The general solutions \eqref{I cases 2 3} and \eqref{I1k} for the operators $I_{\alpha, k}$ are not manifestly $u$-factori-zable because of the factors $(r_i-r_j)^{-1}$. For Case~2, Proposition~\ref{prop I case 2} shows that $I_{\alpha, k}$ is indeed a $u$-factorizable operator (see also \eqref{I case 2 R(u)}).

The explicit expressions of the operators $I_{\alpha, k} \in \B(\hH_k)$ don't give a definitive answer about the final formula: for instance, when applied to an argument containing some $u$'s, the expression can simplify a lot (see for example the proof of Proposition~\ref{prop 1 sur 6 R}). Moreover, the trace and the multiplication introduce some degrees of freedom in the writing of the final expression. This leads us to consider two operators $A$ and $A'$ as {\it equivalent} when
\begin{equation*}
\tr \circ\, m \circ \kappa^* \circ A[B_1 \otimes \cdots \otimes B_k]  = \tr \circ\,  m \circ \kappa^* \circ A'[B_1 \otimes \cdots \otimes B_k] 
\end{equation*}
for any $B_1 \otimes \cdots \otimes B_k \in \H_k$. The equivalence is reminiscent of the lift from $\B(\H_k,M_N)$ to $\B(\hH_k)$ (see Remark \ref{lift}) combined with the trace.

\begin{lemma}
With
\begin{equation*}
\widetilde{I}_{\alpha, k}(r_0, \dots, r_{k-1}) \vc I_{\alpha, k}(r_0, \dots, r_{k-1}, r_0) \quad(=\lim_{r_k \to r_0} I_{\alpha, k}(r_0, \dots, r_{k-1}, r_k) ),
\end{equation*}
the operators $\widetilde{I}_{\alpha, k} \vc \widetilde{I}_{\alpha, k}(R_0(u), \dots, R_{k-1}(u)) = I_{\alpha, k}(R_0(u), \dots, R_{k-1}(u), R_0(u)) \in \B(\hH_k)$ is equivalent with the original $I_{\alpha, k}$.
\end{lemma}

\begin{proof}
For any $B_1 \otimes \cdots \otimes B_k \in \H_k$, using previous notations, one has
\begin{align*}
\tr \circ\, m \circ \kappa^* \circ I_{\alpha, k} [B_1 \otimes \cdots \otimes B_k]
&=
\sum_{i_0,\dots,i_k} I_{\alpha, k}(\lambda_{i_0},\dots,\lambda_{i_k} ) \tr \big( \pi_{i_0} B_1 \pi_{i_1} \cdots B_k \pi_{i_k} \big)
\\
&=
\sum_{i_0,\dots,i_k} I_{\alpha, k}(\lambda_{i_0},\dots,\lambda_{i_k} ) \tr \big( \pi_{i_k}  \pi_{i_0} B_1 \pi_{i_1} \cdots B_k \big)
\\
&=
\sum_{i_0,\dots,i_k} I_{\alpha, k}(\lambda_{i_0},\dots,\lambda_{i_k} )\, \delta_{i_0, i_k} \, \tr \big( \pi_{i_0} B_1 \pi_{i_1} \cdots B_k \big)
\\
&=
\sum_{i_0,\dots,i_{k-1}} \widetilde{I}_{\alpha, k}(\lambda_{i_0},\dots,\lambda_{i_{k-1}} ) \tr \big( \pi_{i_0} B_1 \pi_{i_1} \cdots B_k \big)
\\
&=
\tr \circ\, m \circ \kappa^* \circ \widetilde{I}_{\alpha, k}[B_1 \otimes \cdots \otimes B_k].
\end{align*}
\end{proof}
The equivalence between $\widetilde{I}_{\alpha, k}$ and $I_{\alpha, k}$ seems to be the only generic one we can consider.

We have doubts on the fact that the operators $I_{\alpha,k}$ can be always equivalent to some $u$-factorizable operators. Consider for instance the situation where $M$ is one-dimensional and we want to compute $a_1$ ($d = 1$ and $r=1$). Then, one has
\begin{align*}
&I_{3/2, 2}(r_0, r_1, r_2) = 4 \prod_{0 \leq i < j \leq 2} (\sqrt{r_i}+\sqrt{r_j}\,)^{-1},\\
&\widetilde{I}_{3/2, 2}(r_0, r_1) = 2 \sqrt{r_0}^{\,-1} (\sqrt{r_0}+\sqrt{r_1}\,)^{-2}.
\end{align*}
As shown in next lemma, the function $\widetilde{I}_{\alpha, k}$ cannot be decomposed as in \eqref{F decomposed}, thus by Lemma~\ref{A(u) factorizable}, the operator $\widetilde{I}_{3/2, 2}(R_0(u,R_1(u))$ is not factorizable:  when applied to $\tfrac 12 g_{\mu\nu}v^\mu\otimes v^\nu$ (one of the arguments of $f_2$) it gives the following contribution to $a_1$
\begin{align*}
g_{\mu\nu} \sum_{i_0,i_1} \sqrt{\lambda_{i_0}}^{\,-1} (\sqrt{\lambda_{i_0}} +\sqrt{\lambda_{i_1}}\,)^{-2} \,\tr(\pi_{i_0}v^\mu \pi_{i_1} v^\nu)
\end{align*}
which cannot be cancelled by other terms in the list of Section \ref{subsection-collectcontraction}. This typically generate a non-explicit formula unless $u$ and $v^\mu$ have commutation relations.\\
As a consequence of the existence of roots (see \eqref{recursive}), similar phenomena occur for any odd dimension. We have the same conclusion when $d$ is even: all $a_r$ with $r\geq d/2$ will be given by non-explicit expressions (see \eqref{I1k}). In particular, in dimension 4, $a_2$ will not be explicit despite $a_1$.

\begin{lemma}
The function $(\R^*_+)^2 \to \R^*_+, (x,y) \mapsto (x+y)^{-2}$ cannot be written in the form $\sum_{\text{finite}} h_{(1)}(x) h_{(2)}(y)$ with continuous functions $h_{(i)}$.
\end{lemma}

\begin{proof}
Multiplying by $x+y$, the decomposition $(x+y)^{-2} = \sum_{\text{finite}} h_{(1)}(x) h_{(2)}(y)$ implies $(x+y)^{-1} = \sum_{\text{finite}} (x h_{(1)}(x)) h_{(2)}(y) + \sum_{\text{finite}} h_{(1)}(x) (y h_{(2)}(y))$. So it is sufficient to show that  $(x+y)^{-1}$ cannot be decomposed as $\sum_{\text{finite}} h_{(1)}(x) h_{(2)}(y)$. Let us prove it by contradiction. 

Suppose we have such decomposition $(x+y)^{-1} = \sum_{\ell = 1}^N h_{1, \ell}(x) h_{2, \ell}(y)$ for $N \in \N^*$. Let $(x_i, y_i)_{1 \leq i \leq N}$ be $N$ points in $(\R^*_+)^2$ and consider the $N \times N$-matrix $c_{i,j} \vc (x_i + y_j)^{-1}$. Then\begin{equation*}
\det (c) =\big[\prod_{i,j=1}^N (x_i + y_j) \big]^{-1} \prod_{1 \leq i < j \leq N} (x_i - x_j)(y_i - y_j).
\end{equation*}
This expression shows that we can choose a family $(x_i, y_i)_{1 \leq i \leq N}$ such that $\det (c) \neq 0$. With such a family, define the two matrices $a_{i,\ell} \vc h_{1, \ell}(x_i)$ and $b_{i, \ell} \vc h_{2, \ell}(y_i)$. Then
\begin{align*}
c_{i,j} = (x_i+y_j)^{-1} & = \sum_{\ell = 1}^N h_{1, \ell}(x)\, h_{2, \ell}(y_j) = \sum_{\ell = 1}^N a_{i,\ell}\, b_{j, \ell}
\end{align*}
so that, in matrix notation, $c = a \cdot {}^t b$, which implies that $\det(a) \neq 0$ and $\det(b) \neq 0$.
From $(x+y_j)^{-1} = \sum_{\ell = 1}^N h_{1, \ell}(x)\, b_{j, \ell}\,$, we deduce $h_{1, \ell}(x) = \sum_{j} b^{-1}_{j, \ell} (x+y_j)^{-1}$ and, similarly, $h_{2, \ell}(y) = \sum_{i} a^{-1}_{i, \ell} (x_i+y)^{-1}$. This gives
\begin{equation*}
(x+y)^{-1} = \sum_{i,j,\ell} a^{-1}_{i, \ell} b^{-1}_{j, \ell} \,(x+y_j)^{-1} (x_i+y)^{-1}.
\end{equation*}
This expression must hold true on $(\R^*_+)^2$: when $x,y \to 0^+$, the LHS goes to $+\infty$ while the RHS remains bounded. This is a contradiction.
\end{proof}

{\it A remark about $\zeta_P(0)$}: 
When $r$ is increasing from $0$, the difference $d/2-r$ goes through zero when $d$ is even. Such a point, which appears in \eqref{heat-trace-asympt} to be independent of a dilation of $P$, is the splitting between the pole part of the heat-trace and a (eventually divergent) series in $t$. This splitting is especially interesting because, when $P$ is positive, $a_{d/2}$ is proportional to $\zeta_P(0)$ where $\zeta_P(s)\vc \Tr P^{-s}$ for $\Re(s)$ large enough (see \cite{Gilkeybook}). In that case, Proposition \ref{n<k} shows that for any even-dimensional $M$, $\zeta_P(0)$ contains terms in $\log(u)$ in a non-explicit formula when $u^{\mu\nu}$ is not a multiple of the identity matrix.

%%%%%%%%%%%%%%%%%%%%%%%%%%%%%%%%%%%%%%%%%%%%%%%%%%%%%%%
\subsection{\texorpdfstring{Explicit formulae for scalar symbol ($u(x)=f(x) \bbbone_N$)}{Explicit formulae for scalar symbol}}
%%%%%%%%%%%%%%%%%%%%%%%%%%%%%%%%%%%%%%%%%%%%%%%%%%%%%%%

When $u$ is central, the operator defined by $I_{\alpha,k}(r_0,\dots,r_k)$ is equivalent to the operator defined by 
\begin{align*}
\widetilde I_{\alpha,k}(r_0) \vc \lim_{r_j\to r_0} I_{\alpha,k}(r_0,\dots,r_k) = \tfrac{1}{k!}\, r_0^{-\alpha}.
\end{align*}
Thus \eqref{terme de base} reduces to 
\begin{align*}
\tfrac{1}{(2\pi)^d}\int d\xi \, \xi_{\mu_1}\cdots \xi_{\mu_{2p}} \,f_k(\xi)[\,\BB_k^{\mu_1\dots\mu_{2p}}]& =g_d\, m[u^{-d/2-p} \otimes G(g)_{\mu_1\dots \mu_{2p}} \,\BB_k^{\mu_1\dots \mu_{2p}}]\\
&= g_d\,G(g)_{\mu_1\dots \mu_{2p}} \,u^{-d/2-p}\, m[\BB_k^{\mu_1\dots \mu_{2p}}]
\end{align*}
so the tedious part of the computation is to list all arguments $\BB_k^{\mu_1\dots \mu_{2p}}$ and to contract them with $G(g)_{\mu_1\dots \mu_{2p}} $. This can be done with the help of a computer in any dimension for an arbitrary $r$. All formulae are obviously explicit. An eventual other step is to translate the results in terms of diffeomorphic and gauge invariants.

%%%%%%%%%%%%%%%%%%%%%%%%%%%%%%%%%%%%%%%%%%%%%%%%%%%%%%%
\subsection{Application to quantum field theory}
\label{AppliQFT}
%%%%%%%%%%%%%%%%%%%%%%%%%%%%%%%%%%%%%%%%%%%%%%%%%%%%%%%

Second-order differential operators which are on-minimal have a great importance in physics and have justified their related heat-trace coefficients computation (to quote but a few see almost all references given here). For instance in the interesting work \cite{MT,Toms}, the operators $P$ given in \eqref{def-P} are investigated under the restriction
\begin{align*}
u^{\mu\nu}=g^{\mu\nu} \,\bbbone+\zeta\,X^{\mu\nu},
\end{align*}
where $\zeta$ is a parameter (describing for $\zeta=0$ the minimal theory), under the constraints for the normalized symbol $\widehat{X}(\sigma) \vc X^{\mu\nu} \sigma_\mu\sigma_\nu$ with $\vert \sigma\vert_g^2=g^{\mu\nu} \sigma_\mu\sigma_\nu=1$ given by
\begin{align}
& \widehat{X}(\sigma)\,^2=\widehat{X}(\sigma), \text{ for any }\sigma \in S^1_g\,, \label{X2=X}\\
& \nabla_\rho\, X^{\mu\nu}=0.\label{parallel}
\end{align}
Here, $\nabla_\rho$ is a covariant derivative involving gauge and Christoffel connections. In covariant form, the operators are
$$
P=-(g^{\mu\nu}\,\nabla_\mu\nabla_\nu + \zeta \,X^{\mu\nu}\,\nabla_\mu\nabla_\nu +Y).
$$
Despite the restrictions \eqref{X2=X}-\eqref{parallel} meaning that the operator $\widehat{X}$ is a projector and the tensor-endomorphism $u$ is parallel, this covers an operator describing a quantized spin-1 vector field
$$
{P^\mu}_\nu =-({\delta^\mu}_\nu \,\nabla^2+\zeta \,\nabla^\mu\nabla_\nu + {Y^\mu}_\nu),
$$
or a Yang--Mills field
$$
{P^\mu}_\nu =-({\delta^\mu}_\nu \,D^2 -\zeta\, D^\mu D_\nu +{R^\mu}_\nu -2 {F^\mu}_\nu)
$$
where $D_\mu\vc \nabla_\nu + A_\mu$ and $A_\mu,\,F_{\mu\nu}$ are respectively the gauge and strength fields, \\
or a perturbative gravity (see \cite{MT} for details). 

Remark first that 
\begin{align*}
& H(x,\xi)=u^{\mu\nu} \xi_\mu \xi_\nu= \vert \xi \vert_g^2\,[1+\zeta\, \widehat{X}\,(\xi/ \vert \xi \vert_g)], \text{ so that } \\
& e^{-H(x,\xi)} = [e^{-(1+\zeta) \vert \xi \vert_g^2}-e^{-\vert \xi \vert_g^2}]\, \widehat{X} + e^{-\vert \xi \vert_g^2} \,\bbbone_V.
\end{align*}
Thus \eqref{a0(x)} becomes
\begin{align*}
a_0(x)& = \tfrac{1}{(2\pi)^d} \int d\xi\,[e^{-(1+\zeta )\vert \xi \vert_g^2}-e^{-\vert \xi \vert_g^2}]\, \tr \widehat{X} + \tfrac{1}{(2\pi)^d} \int e^{-\vert \xi \vert_g^2} \tr \,\bbbone_V \\
& =  [ \int_{\sigma \in S^1_g} d\Omega(\sigma) \,\tr \widehat{X}(\sigma)]\,  [\tfrac{1}{(2\pi)^d}\int_0^\infty dr\,r^{d-1}\,(e^{-(1+\zeta) r^2}-e^{-r^2})] + g_d\,N \\
&= \tfrac{\Gamma(d/2)}{2(2\pi)^d}\,[(1+\zeta)^{-d/2} -1]\,[ \int_{\sigma \in S^1_g} d\Omega(\sigma) \,\tr \widehat{X}(\sigma)]+ g_d\,N .
\end{align*}
One has $g_d \vc \tfrac{1}{(2\pi)^d} \int_{\R^d} d\xi\, e^{-\abs{\xi}_{g(x)}^2}= \tfrac{\abs{g}^{1/2}}{2^{d}\,\pi^{d/2}}$ and
\begin{align*}
\int_{\sigma \in S^1_g} d\Omega(\sigma) \,\tr \widehat{X}(\sigma) &=
\tr(X^{\mu\nu})\,  \int_{\sigma \in S^1_g} d\Omega(\sigma) \, \sigma_\mu \sigma_\nu
\end{align*}
Using \eqref{G(g)}, we can get
\begin{align*}
\int_{\sigma \in S^1_g} d\Omega(\sigma) \, \sigma_{\mu_1} \cdots \sigma_{\mu_{2p}} = \tfrac{2 \, (2 \pi)^d \, g_d}{\Gamma(d/2 + p)} \, G(g)_{\mu_1 \dots \mu_{2p}},
\end{align*}
so that we recover \cite[(2.34)]{MT}:
\begin{align*}
a_0(x) = g_d \left[  N + d^{-1}\, g_{\mu\nu} \tr(X^{\mu\nu}) \big( (1+\zeta)^{-d/2} -1 \big)\right].
\end{align*}

Now, let us consider the more general case 
\begin{align}
\label{u+X}
u^{\mu\nu} = g^{\mu\nu}u + \zeta \,X^{\mu\nu} 
\end{align}
where $u$ is a strictly positive matrix $u(x) \in M_N$ as in Section~\ref{Section-example}, $X^{\mu\nu}$ as before and assume $[u(x), X^{\mu\nu}(x)] = 0$ for any $x \in (M,g)$ and $\mu, \nu$. Previous situation was $u = \bbbone_V$, the unit matrix in $M_N$. Once Lemma~\ref{fk avec derivee} has been applied, the difficulty to compute $a_r(x)$ is to evaluate the operators  $T_{k,p}$ defined by \eqref{Tkp}. Here we have $u[\sigma] = u + \zeta \,\widehat{X}(\sigma)$, where the two terms commute. With the notation $$\widehat{X}_i \vc R_i[\widehat{X}(\sigma)],$$
we have
\begin{align*}
C_k(s,u[\sigma]) &= (1-s_1)R_0[u + \zeta \,\widehat{X}(\sigma)] + (s_1-s_2)R_1[u + \zeta \,\widehat{X}(\sigma)] + \dots \\
& \hspace{2cm} \dots + (s_{k-1} - s_k)R_{k-1}[u + \zeta\, \widehat{X}(\sigma)] + s_k R_k[u + \zeta\, \widehat{X}(\sigma)] \\
& = C_k(s,u) + \zeta \big[ (1-s_1) \widehat{X}_0 + (s_1-s_2) \widehat{X}_1 + \dots + (s_{k-1} - s_k) \widehat{X}_{k-1} + s_k \widehat{X}_k \big]
\end{align*}
so that, using the fact that each $\widehat{X}_i$ is a projection with eigenvalues $\epsilon_i = 0,1$:
\begin{align*}
e^{-r^2 C_k(s,u[\sigma])} &=
e^{-r^2 C_k(s,u)} \sum_{(\epsilon_i) \in \{0,1\}^{k+1}} e^{-\zeta r^2 [ (1-s_1) \epsilon_0 + (s_1-s_2) \epsilon_1 + \dots + (s_{k-1} - s_k) \epsilon_{k-1} + s_k \epsilon_k ]} \times \\
&\hspace{4cm} 
\times {\widehat{X}_0}^{\epsilon_0} (1-\widehat{X}_0)^{1-\epsilon_0} \cdots {\widehat{X}_k}^{\epsilon_k} (1-\widehat{X}_k)^{1-\epsilon_k}.
\end{align*}
Notice that $\widehat{X}_i^{\epsilon_i} (1-\widehat{X}_i)^{1-\epsilon_i} = [(1-\epsilon_i) g^{\mu\nu} + (2\epsilon_i - 1) R_i(X^{\mu\nu}) ] \sigma_\mu \sigma_\nu$.
\\
With the definition
%\begin{multline*}
%I_{\alpha,k}^{\{\epsilon_i\}, \zeta}(r_0,r_1,\dots, r_k)  \vc \\
%\int_{\Delta_k} ds\, [(1-s_1)(r_0 + \zeta \epsilon_0)+(s_1-s_2)(r_1 + \zeta \epsilon_1)+\dots +s_k (r_k + \zeta \epsilon_k)]^{-\alpha}
%\end{multline*}
\begin{align*}
I_{\alpha,k}^{(\epsilon_i),\, \zeta}(r_0,r_1,\dots, r_k) \vc I_{\alpha,k}(r_0 + \zeta \epsilon_0 ,r_1 + \zeta \epsilon_1, \dots, r_k + \zeta \epsilon_k),
\end{align*}
the operators $T_{k,p}$ of \eqref{Tkp versus Ink} become
\begin{align}
T_{k,p}= &\tfrac{\Gamma(d/2+p)}{2(2\pi)^d} \int_{S^{d-1}_g} d\Omega_g(\sigma)\,\sigma_{\mu_1}\cdots \sigma_{\mu_{2p}} \sigma_{\alpha_0} \sigma_{\beta_0} \cdots  \sigma_{\alpha_k} \sigma_{\beta_k}  \times \nonumber
\\
& \sum_{(\epsilon_i) \in \{0,1\}^{k+1}} \, I_{d/2+p,k}^{(\epsilon_i), \,\zeta} \big(R_0(u),R_1(u),\dots,R_k(u)\big) \times \nonumber
\\
& \times [(1-\epsilon_0) g^{\alpha_0\beta_0} + (2\epsilon_0 - 1) R_0(X^{\alpha_0\beta_0}) ] \cdots [(1-\epsilon_k) g^{\alpha_k\beta_k} + (2\epsilon_k - 1) R_k(X^{\alpha_k\beta_k}) ]. \label{Tkpadapte}
\end{align}
The computations of these operators are doable using the methods given in Section~\ref{Section-example}, but with some more complicated combinatorial expressions requiring a computer. For $a_1(x)$ in $d=4$ we will still get explicit formulae. The main combinatorial computation is to make the contractions between $G(g)_{\mu_1\dots \mu_{2p}\alpha_0\beta_0\dots \alpha_k\beta_k}$ from the first line of \eqref{Tkpadapte} with the operators of the last line applied on variables $\BB_k^{\mu_1\dots \mu_{2p}}$.\\
When $u=\bbbone_V$ in \eqref{u+X}, the operators in the second line of \eqref{Tkpadapte} are just the multiplication by the numbers $I_{d/2+p,k}(1 + \zeta \epsilon_0 ,1 + \zeta \epsilon_1, \dots, 1 + \zeta \epsilon_k)$. Thus one gets explicit formulae for all coefficients $a_r(x)$ in any dimension.

%%%%%%%%%%%%%%%
\section{Conclusion}
%%%%%%%%%%%%%%%

On the search of heat-trace coefficients  for Laplace type operators $P$ with non-scalar symbol, we develop, using functional calculus, a method where we compute some operators $T_{k,p}$ acting on some (finite dimensional) Hilbert space and the arguments on which there are applied. This splitting allows to get general formulae for these operators and so, after a pure computational machinery will yield all coefficients $a_r$ since there is no obstructions other than the length of calculations. The method is exemplified when the principal symbol of $P$ has the form $g^{\mu\nu} u$ where $u$ is a positive invertible matrix. It gives $a_1$ in dimension 4 which is written both in terms of ingredients of $P$ (analytic approach) or of diffeomorphic and gauge invariants (geometric approach). As just said, the method is yet ready for a computation of $a_r$ with $r\geq 2$ for calculators patient enough, as well for the case $g^{\mu\nu}u + \zeta \,X^{\mu\nu}$ as in Section \ref{AppliQFT}. Finally, the method answers a natural question about explicit expressions for all coefficients $a_r$. When the dimension is odd we saw that $u$-factorizability is always violated and is preserved in even dimension $d$ only when $d/2-r>0$. Thus we conjecture that for any dimension $d$, there exists coefficients $a_r$ which have no explicit “$u$-factorizable” formulae when $u^{\mu\nu} = g^{\mu\nu} u$. In particular, in dimension 4, the problem appears already with the computation of $a_2$ despite the nice formulae \eqref{valeur de a1} or \eqref{valeur de a1 en u pmu q} obtained for $a_1$.

%%%%%%%%%%%%%%%%%%%%%%%%%%%%%%%%%%%%%%%
\appendix
\section{Appendix}
%%%%%%%%%%%%%%%%%%%%%%%%%%%%%%%%%%%%%%%

\subsection{Some algebraic results}

Let $A$ be a unital associative algebra over $\C$, with unit $\bbbone$. \\
Denote by $(C^*(A,A) = \oplus_{k \geq 0} C^k(A,A), \delta)$ the Hochschild complex where $C^k(A,A)$ is the space of linear maps $\omega : A^{\otimes^k} \to A$ and
\begin{multline*}
(\delta \omega) [b_0 \otimes \cdots \otimes b_k] = b_0 \omega[b_1 \otimes \cdots \otimes b_k]
+ \sum_{i=0}^k (-1)^i \omega[b_0 \otimes \cdots \otimes b_{i-1} \otimes b_{i+1} \otimes \cdots \otimes b_k]\\
+ (-1)^{k+1} \omega[b_0 \otimes \cdots \otimes b_{k-1}] b_k
\end{multline*}
for any $\omega \in C^k(A,A)$ and $b_0 \otimes \cdots \otimes b_k \in A^{\otimes^k}$.  

Define the differential complex $(\mathfrak{T}^* A = \oplus_{k \geq 0} A^{\otimes^{k+1}}, d)$ with 
\begin{multline*}
d (a_0 \otimes \cdots \otimes a_k) = \bbbone \otimes a_0 \otimes \cdots \otimes a_k
+ \sum_{i=0}^k (-1)^i a_0 \otimes \cdots \otimes a_{i-1} \otimes \bbbone \otimes a_{i+1} \otimes \cdots \otimes a_k\\
+ (-1)^{k+1} a_0 \otimes \cdots \otimes a_k \otimes \bbbone 
\end{multline*}
for any $a_0 \otimes \cdots \otimes a_k \in A^{\otimes^{k+1}}$. Both $C^*(A,A)$ and $\mathfrak{T}^* A$ are graded differential algebras, the first one for the product
\begin{equation*}
(\omega \,\omega')[b_1 \otimes \cdots \otimes b_{k+k'}] = \omega[b_1 \otimes \cdots \otimes b_{k}] \,\omega'[b_{k+1} \otimes \cdots \otimes b_{k+k'}]
\end{equation*}
and the second one for 
\begin{equation*}
(a_0 \otimes \cdots \otimes a_k)(a'_0 \otimes \cdots \otimes a'_{k'})
= a_0 \otimes \cdots \otimes a_k a'_0 \otimes \cdots \otimes a'_{k'} \in A^{\otimes^{k + k'+1}} = \mathfrak{T}^{k+k'} A.
\end{equation*}
The following result was proved in \cite{Mass95}:

\begin{proposition}
\label{prop-iso}
The map $\iota : \mathfrak{T}^* A \to C^*(A,A)$ defined by
\begin{equation*}
\iota(a_0 \otimes \cdots \otimes a_k)[b_1 \otimes \cdots \otimes b_k] \vc a_0 b_1 a_1 b_2 \cdots b_k a_k
\end{equation*}
is a morphism of graded differential algebras.

If $A$ is central simple, then $\iota$ is injective, and if $A = M_N$ (algebra of $N \times N$ matrices) then $\iota$ is an isomorphism.
\end{proposition}
Recall that an associative algebra $A$ is central simple if it is simple and its center is $\C$. Central simple algebras have the following properties, proved for instance in \cite{Lam91a}:

\begin{lemma}
\label{lemma-centralsimple}
If $B$ is a central simple algebra and $C$ is a simple algebra, then $B\otimes C$ is a simple algebra. If moreover $C$ is central simple, then $B \otimes C$ is also central simple.
\end{lemma}

\begin{proof}[of Proposition~\ref{prop-iso}]
A pure combinatorial argument shows that $\iota$ is a morphism of graded differential algebras for the structures given above.

Assume that $A$ is central simple. The space $A^{\otimes^{k+1}}$ is an associative algebra for the product $(a_0 \otimes \cdots \otimes a_k)\cdot(a'_0 \otimes \cdots \otimes a'_k) = a_0 a'_0 \otimes \cdots \otimes a_k a'_k$, which is central simple by Lemma~\ref{lemma-centralsimple}. Let $J_k= \Ker \iota \cap A^{\otimes^{k+1}}$. Then, for any $\alpha = \sum_i a_{0,i} \otimes \cdots \otimes a_{k,i} \in J_k$, any $\beta = b_0 \otimes \cdots \otimes b_k \in A^{\otimes^{k+1}}$, and any $c_1 \otimes \cdots \otimes c_k \in A^{\otimes^{k}}$, one has
\begin{align*}
\iota(\alpha \cdot \beta)[c_1 \otimes \cdots \otimes c_k]
&= \sum_i a_{0,i} b_0 c_1 a_{1,i} b_1 c_2 \cdots b_{k-1} c_k a_{k,i} b_k = \iota(\alpha)[b_0 c_1 \otimes \cdots \otimes b_{k-1} c_k] b_k = 0,
\end{align*}
so that $\alpha\cdot \beta \in J_k$. The same argument on the left shows that $J_k$ is a two-sided ideal of the algebra $A^{\otimes^{k+1}}$, which is simple. Since $\iota$ is non zero ($\iota(\bbbone \otimes \cdots \otimes \bbbone) \neq 0$), one must have $J_k = 0$: this proves that $\iota$ is injective.

The algebra $A = M_N$ is central simple \cite{Lam91a}, so that $\iota$ is injective, and moreover the spaces $C^k(A,A)$ and $A^{\otimes^{k+1}}$ have the same dimensions: this shows that $\iota$ is an isomorphism.
\end{proof}

\begin{remark}
In \cite{Mass95}, the graded differential algebras $C^*(A,A)$ and $\mathfrak{T}^* A$ are equipped with a natural Cartan operations of the Lie algebra $A$ (where the bracket is the commutator) and it is shown that $\iota$ intertwines these Cartan operations.
\end{remark}

%%%%%%%%%%%%%%%%%%%%%%%
\subsection{Some combinatorial results}
%%%%%%%%%%%%%%%%%%%%%%%

\begin{lemma}
\label{Vandermonde}
Given a family $a_0,\dots,a_r$ of different complex numbers, we have
\begin{align}
&i)\quad \sum_{n=0}^r \, a_n^s\,\prod^r_{m=0,\,m\neq n} (a_n-a_m)^{-1} = 0,\quad \text{for any } s \in \{0,1,\dots,r-1\},\label{eq-appendix 1}\\
&ii) \quad  \prod_{m=0}^r (z-a_m)^{-1} =\sum_{n=0}^r \,(z-a_n)^{-1} \prod_{m=0,\,m\neq n}^r (a_n-a_m)^{-1}, \quad \forall z \in \C\backslash \{a_0,\dots,a_r\} . \label{eq-appendix 2}
\end{align}
\end{lemma}

\begin{proof}
i) If $\alpha(s) \vc \sum_{n=0}^r a_n^s\,\prod^r_{m=0,\,m\neq n} (a_n-a_m)^{-1}$ and 
$$
\beta(s)\vc \alpha(s) \prod_{0\leq l<k\leq r}(a_k-a_l) = \sum_{n=0}^r (-1)^{r-n}\,a_n^s \,\prod^r_{\substack{0\leq l<k\leq r\\ k \neq n,\, l\neq n}} (a_k-a_l)
$$
then it is sufficient to show that $\beta(s)=0$ for $s=0,\dots,r-1$. Recall first that the determinant of a Vandermonde $p\times p$-matrix is
\begin{align*}
\det \left(\begin{array}{cccc}
1&1&\cdots&1\\
b_1&b_2&\cdots &b_p\\
b_1^2&b_2^2&\cdots &b_p^2 \\
\vdots&\vdots&\cdots&\vdots\\
b_1^{p-1}&b_2^{p-1} &\cdots& b_p^{p-1} 
\end{array} \right)
= \prod_{1\leq j <i\leq p}(b_i-b_j).
\end{align*}
Thus
\begin{align*}
\det \left(\begin{array}{cccc}
1&1&\cdots&1\\
a_0&a_1&\cdots &a_r\\
a_0^2&a_1^2&\cdots &a_r^2 \\
\vdots&\vdots&\cdots&\vdots\\
a_0^{r-1}&a_1^{r-1} &\cdots& a_r^{r-1} \\
a_0^s&a_1^s&\cdots & a_r^s
\end{array} \right)
=\beta(s)
\end{align*}
after an expansion of the determinant with respect to the last line. 
But this is zero since the last line coincides with the line $s+1$ of the matrix when $s=0,\dots,r-1$.

ii) The irreducible fraction expansion of $f(z)=\tfrac{1}{(z-a_0)(z-a_1)\cdots(z-a_r)}$ is $\sum_{n=0}^r \tfrac{ \text{Res} (f)(a_n)}{z-a_n}$ yielding \eqref{eq-appendix 2}.
\end{proof}

%%%%%%%%%%%%%%%%%%%%%%%%%%%%%%%%%%%%%
\subsection{\texorpdfstring{Diffeomorphism invariance and gauge covariance of $P$}{Diffeomorphism invariance and gauge covariance of P}}
\label{subsection-diff-gauge-P}
%%%%%%%%%%%%%%%%%%%%%%%%%%%%%%%%%%%%%

The operator $P$ in \eqref{def-P} is given in terms of $u^{\mu\nu},v^\mu,w$, which are not well adapted to study changes of coordinates and gauge transformations. 

Given a change of coordinates (c.c.) $x\tocc x'$, let $J_\mu^\nu \vc \pmu x'^\mu$ and $\abs{J}\vc \det (J_\mu^\nu)$, so that 
\begin{align*}
\pmu &\tocc {J^{-1}}^\alpha_\mu \partial_\alpha,
&
g_{\mu\nu} &\tocc {J^{-1}}_\mu^\rho {J^{-1}}_\nu^\sigma g_{\rho \sigma},
&
g^{\mu\nu} &\tocc {J}^\mu_\rho {J}^\nu_\sigma g^{\rho \sigma},
&
\abs{g} &\tocc \abs{J}^{-2} \abs{g}
\end{align*}
where $\abs{g} = \det(g_{\mu\nu})$. Denote by $\gamma$ a gauge transformation, local in the trivialization of $V$ in which $P$ is written (\textsl{i.e.} $\gamma$ is a map from the open subset of $M$ which trivializes $V$ to the gauge group $GL_N(\C)$). A section $s$ of $V$ performs the transformations $s \tocc s$ (diffeomorphism-invariant) and $s \togt \gamma s$ (where  “g.t.” stands for “gauge transformation”). \\The proof of the following lemma is a straightforward computation:

\begin{lemma}
\label{lem-u vmu w cc gt}
The differential operator $P$ is well defined on sections of $V$ if and only if $u^{\mu\nu}$, $v^\mu$ and $w$ have the following transformations:
\begin{align}
u^{\mu\nu} &\tocc J_\rho^\mu J_\sigma^\nu u^{\rho\sigma}, 
& 
u^{\mu\nu} &\togt \gamma u^{\mu\nu}\gamma^{-1}, 
\\
v^\mu &\tocc J_\rho^\mu v^\rho + (\psigma J_\rho^\mu) u^{\rho\sigma}, 
& 
v^\mu &\togt \gamma v^\mu \gamma^{-1} +2 \gamma u^{\mu\nu} (\pnu \gamma^{-1}), 
\label{eq-vmu cc gt}
\\
w &\tocc w,
 &
 w &\togt \gamma w \gamma^{-1} + \gamma v^\mu(\pmu \gamma^{-1}) +\gamma u^{\mu\nu} (\pmu\pnu \gamma^{-1}).
\end{align}
\end{lemma}
As these relations show, neither $v^\mu$ nor $w$ have simple transformations under changes of coordinates or gauge transformations.

It is possible to parametrize $P$ using structures adapted to changes of coordinates and gauge transformations. Let us fix a (gauge) connection $A_\mu$ on $V$ and denote by $\nmu \vc \pmu + A_\mu$ its associated covariant derivative on sections of $V$. For any section $s$ of $V$, one then has:
\begin{align}
A_\mu &\tocc {J^{-1}}^\alpha_\mu A_\alpha,
&
A_\mu &\togt \gamma A_\mu \gamma^{-1} + \gamma (\pmu \gamma^{-1}),
&
\nmu s & \tocc {J^{-1}}^\alpha_\mu \nabla_\alpha s,
&
\nmu s & \togt \gamma \nmu s.
\label{eq-ccgtANabla}
\end{align}

\begin{lemma}
The differential operator 
\begin{equation}
s \mapsto Q s \vc -(\abs{g}^{-1/2} \nmu \abs{g}^{1/2} u^{\mu\nu} \nnu +p^\mu \nmu +q)\, s \label{Qs}
\end{equation}
is well defined on sections of $V$ if and only if $u^{\mu\nu}$, $p^\mu$ and $q$ have the following transformations:
\begin{align}
u^{\mu\nu} &\tocc J_\rho^\mu J_\sigma^\nu u^{\rho\sigma}, 
&
p^\mu &\tocc J_\rho^\mu p^\rho, 
&
q &\tocc q,
\label{eq-ccupq}\\
u^{\mu\nu} &\togt \gamma u^{\mu\nu}\gamma^{-1}, 
& 
p^\mu &\togt \gamma p^\mu \gamma^{-1}, 
 &
q &\togt \gamma q \gamma^{-1}.
\label{eq-gtupq}
\end{align}
It is equal to $P$ when (the $u^{\mu\nu}$ are the same in $P$ and $Q$)
\begin{align}
& v^\mu = \tfrac 12 (\pnu \log \abs{g})u^{\mu\nu} + \pnu u^{\mu\nu} + u^{\mu\nu} A_\nu + A_\nu u^{\mu\nu} + p^\mu,\label{vmu(A,p,q)}  \\
& w = \tfrac 12 (\pmu \log \abs{g})u^{\mu\nu} A_\nu + (\pmu u^{\mu\nu}) A_\nu + u^{\mu\nu} (\pmu A_\nu) + A_\mu u^{\mu\nu} A_\nu  + p^\mu A_\mu + q. \label{w(A,p,q)}
\end{align}
\end{lemma}

\begin{proof}
Combining \eqref{eq-ccgtANabla}, \eqref{eq-ccupq} and \eqref{eq-gtupq},  the operator $s \mapsto -(p^\mu \nmu +q)\, s$ is well behaved under changes of coordinates and gauge transformations. Let $X\vc \abs{g}^{-1/2} \nmu \abs{g}^{1/2} u^{\mu\nu} \nnu$ (a matrix valued “Laplace--Beltrami operator”). Then, using \eqref{eq-ccgtANabla} and \eqref{eq-ccupq}, one has
\begin{align*}
X \tocc \,& 
\abs{J}\abs{g}^{-1/2} {J^{-1}}^\alpha_\mu \nabla_\alpha \abs{J}^{-1}\abs{g}^{1/2} {J}^\mu_\rho {J}^\nu_\sigma u^{\rho\sigma}{J^{-1}}_\nu^\beta \nabla_\beta\\
&= 
\abs{J}\abs{g}^{-1/2} {J^{-1}}^\alpha_\mu \partial_\alpha \abs{J}^{-1}\abs{g}^{1/2} {J}^\mu_\rho u^{\rho\sigma} \nabla_\sigma + {J^{-1}}_\mu^\alpha J_\rho^\mu A_\alpha u^{\rho\sigma} \nabla_\sigma
\\
& 
= 
\abs{J} {J^{-1}}_\mu^\alpha J_\rho^\mu  (\partial_\alpha \abs{J}^{-1}) u^{\rho\sigma} \nabla_\sigma 
+ \abs{g}^{-1/2} {J^{-1}}_\mu^\alpha J_\rho^\mu  (\partial_\alpha \abs{g}^{1/2}) u^{\rho\sigma} \nabla_\sigma 
\\
&\quad\quad 
+ {J^{-1}}_\mu^\alpha(\partial _\alpha J_\rho^\mu) u^{\rho\sigma} \nabla_\sigma 
+{J^{-1}}_\mu^\alpha J_\rho^\mu \partial_\alpha u^{\rho\sigma} \nabla_\sigma 
+ A_\rho u^{\rho\sigma} \nabla_\sigma 
\\
& = 
\nabla_\rho u^{\rho\sigma} \nabla_\sigma
+ \abs{g}^{-1/2} (\prho \abs{g}^{1/2})u^{\rho\sigma} \nabla_\sigma  \\
&\quad\quad 
+ \abs{J}(\prho\abs{J}^{-1}) u^{\rho\sigma} \nabla_\sigma  
+ {J^{-1}}_\mu^\alpha (\partial_\alpha J_\rho^\mu) u^{\rho\sigma} \nabla_\sigma  
\end{align*}
which is equal to $X$ since $\abs{J}\prho \abs{J}^{-1}= -\abs{J}^{-1}\prho\abs{J}=-\prho \log \abs{J}=-{J^{-1}}_\mu^\alpha(\prho J_\alpha^\mu)$ (Jacobi's formula) and $\partial_\alpha J_\rho^\mu = \partial_\alpha \prho x'^\mu = \prho \partial_\alpha x'^\mu = \prho J_\alpha^\mu$.
Similarly, using \eqref{eq-ccgtANabla} and \eqref{eq-gtupq}, one has
\begin{equation*}
\abs{g}^{1/2} u^{\mu\nu} \nnu s \togt \gamma \abs{g}^{1/2} u^{\mu\nu} \nnu s \quad \text{(it behaves like a section of $V$)},
\end{equation*}
so that $Xs \togt  \gamma Xs$. This proves that $Q = -X  -(p^\mu \nmu +q)$ is well defined.

The expansion of $X$ gives
\begin{align*}
X 
& = 
\abs{g}^{-1/2} \pmu \abs{g}^{1/2} u^{\mu\nu} (\pnu 
+ A_\nu) + A_\nu u^{\mu\nu}(\pnu+ A_\nu) 
\\
&= 
\abs{g}^{-1/2} (\pmu \abs{g}^{1/2}) u^{\mu\nu} \pnu 
+ \abs{g}^{-1/2} (\pmu \abs{g}^{1/2}) u^{\mu\nu} A_\nu
+ (\pmu u^{\mu\nu}) \pnu
+ (\pmu u^{\mu\nu}) A_\nu
\\
& \quad\quad
+ u^{\mu\nu} \pmu \pnu  
+ u^{\mu\nu} (\pmu A_\nu)
+  u^{\mu\nu} A_\nu \pmu
+ A_\mu u^{\mu\nu} \pnu 
+ A_\mu u^{\mu\nu} A_\nu 
\\
&=
u^{\mu\nu} \pmu \pnu
+\tfrac 12 (\pmu \log \abs{g})u^{\mu\nu} \pnu
+ (\pmu u^{\mu\nu}) \pnu
+  u^{\mu\nu} A_\nu \pmu
+ A_\mu u^{\mu\nu} \pnu 
\\
& \quad\quad
+\tfrac 12 (\pmu \log \abs{g})u^{\mu\nu} A_\nu
+ (\pmu u^{\mu\nu}) A_\nu
+ u^{\mu\nu} (\pmu A_\nu)
+ A_\mu u^{\mu\nu} A_\nu 
\end{align*}
which, combined with the contributions of $-(p^\mu \nmu +q)$, gives \eqref{vmu(A,p,q)} and \eqref{w(A,p,q)}.
\end{proof}
Contrary to the situation in \cite[Section 1.2.1]{Gilkeybook}, one cannot take directly $p=0$ in \eqref{Qs} since we cannot always solve $A_\mu$ in \eqref{vmu(A,p,q)} to write it in terms of $u^{\mu\nu},\,v^\mu,\,w$.

%%%%%%%%%%%%%%%%%%%%%%%%%%%%%%%%%%%%%%%
\section*{Acknowledgments}
%%%%%%%%%%%%%%%%%%%%%%%%%%%%%%%%%%%%%%%

We would like to thank Andrzej Sitarz and Dimitri Vassilevich for discussions during the early stage of this work and Thomas Krajewski for his help with Lemma \ref{Vandermonde}.

%%%%%%%%%%%%%%%%%%%%%%%%%%%%%%%%%%%%%%%

\bigskip

\noindent bruno.iochum@cpt.univ-mrs.fr\\
thierry.masson@cpt.univ-mrs.fr\\
Centre de Physique Théorique, \\ 
Campus de Luminy, Case 907, 
163 Avenue de Luminy, 
13288 Marseille cedex 9, France

\end{document}